\documentclass[11pt,letterpaper]{amsart}

\usepackage{amsmath,amssymb,amsfonts,amsthm,enumitem,xy,pstricks,pst-node,comment,mathrsfs,tikz}
\xyoption{all}
\usepackage{footnote}
\usepackage[symbol]{footmisc}

\usepackage[initials,nobysame]{amsrefs}

\BibSpec{collection.article}{%
	+{}  {\PrintAuthors}				{author}
	+{,} { \textit}                  	{title}
	+{.} { }                         	{part}
	+{:} { \textit}                  	{subtitle}
	+{,} { \PrintContributions}      	{contribution}
	+{,} { \PrintConference}         	{conference}
	+{}  {\PrintBook}                	{book}
	+{,} { }                         	{booktitle}
	+{,} { }						 	{series}
	+{,} { }						 	{publisher}
	+{,} { }						 	{address}
	+{,} { \PrintDateB}              	{date}
	+{,} { pp.~}                     	{pages}
	+{,} { }                         	{status}
	+{,} { \PrintDOI}                	{doi}
	+{,} { available at \eprint}     	{eprint}
	+{}  { \parenthesize}            	{language}
	+{}  { \PrintTranslation}        	{translation}
	+{;} { \PrintReprint}            	{reprint}
	+{.} { }                         	{note}
	+{.} {}                          	{transition}
	+{}  {\SentenceSpace \PrintReviews} {review}
}

\numberwithin{equation}{section}
\newtheorem{thm}{Theorem}[section] 
\newtheorem{prp}[thm]{Proposition}
\newtheorem{lmm}[thm]{Lemma}   
\newtheorem{crl}[thm]{Corollary} 

\theoremstyle{definition}
\newtheorem{dfn}[thm]{Definition}
\newtheorem{rmk}[thm]{Remark}
\newtheorem{eg}[thm]{Example}
\newtheorem{ques}{Question}

\def\BE#1{\begin{equation}\label{#1}}
\def\EE{\end{equation}}
\def\e_ref#1{(\ref{#1})}

\def\ov#1{\overline{#1}}

\def\wt#1{\widetilde{#1}}

\def\lra{\longrightarrow}

\def\al{\alpha}
\def\be{\beta}

\def\C{\mathcal C}

\def\R{\mathbb R}

\def\Z{\mathbb Z}

\def\ov#1{\overline{#1}}

\def\wt#1{\widetilde{#1}}

\def\lra{\longrightarrow}

\def\C{\mathbb C}

\def\R{\mathbb R}

\def\Z{\mathbb Z}
\def\avg{\textnormal{Avg}}

\def\al{\alpha}
\def\be{\beta}

\def\relbound{20}

\newcommand\res{\mathop{\hbox{\vrule height 7pt width .3pt depth 0pt
\vrule height .3pt width 5pt depth 0pt}}\nolimits}

\title{Ahlfors Currents and Symplectic Non-Hyperbolicity}
\author{Spencer Cattalani}
\thanks{Partially supported by NSF grants DMS 2246485 and DMS 1901979 and by the Simons Foundation}
\date{\today}

\begin{document}

\maketitle
%\vspace{-.1in}
\begin{abstract}
Complex (affine) lines are a major object of study in complex geometry, but their symplectic aspects are not well understood. Inspired by Duval's work on Ahlfors currents, we use them to perform a systematic study of complex lines in symplectic manifolds. In particular, we generalize (by a different method and under topological assumptions) a result of Bangert on the existence of complex lines. We show that Ahlfors currents control the asymptotic behavior of families of pseudoholomorphic curves, refining a result of Demailly. Lastly, we show that the space of Ahlfors currents is convex.
\end{abstract}

\tableofcontents

\section{Introduction}
Symplectic geometry has benefitted greatly from the study of pseudoholomorphic curves. In particular, the study of compact pseudoholomorphic curves has opened connections to complex geometry and revealed much algebraic structure in the symplectic category (e.g. Gromov-Witten invariants, the Fukaya category, Floer theory, etc.) However, there are transcendental phenomena not detected by compact pseudoholomorphic curves. For example, although all Gromov-Witten invariants vanish for tori, Bangert has proved the following theorem:

\begin{thm}[\cite{Ban98}]\label{Bangert_thm}
There exists a nonconstant $J$-holomorphic map $u: \C \lra T^{2n}$ for any almost complex structure $J$ tamed by the standard symplectic structure on $T^{2n}$.
\end{thm}

In complex geometry, a manifold which lacks holomorphic lines is called \textit{hyperbolic}\footnote{There are different notions of hyperbolicity in the literature. This one differs from the notion of symplectic hyperbolicity in \cite{Bio04}, which is Floer-theoretic.}. The study of complex hyperbolicity is quite rich, and many aspects can be seen to be symplectic. 
For instance, \cites{Duv04, Duv19, Sal14} generalize classic results to the symplectic setting. However, the symplectic topology of complex lines remains obscure. Sikorav~\cite{Sik98} and Gromov~\cite[Section IV]{Gro00} have both noted the potential for a theory of symplectic (non)hyperbolicity. Much work has been directed toward this goal \cites{Bio04, Deb01, Hag17, Hag10, hindhyperbolicity, Iva04, Kob01, Li12}.\\

Due to their noncompactness, it is difficult to make quantitative statements about complex lines. In fact, in the words of McQuillan~\cite{McQ02}: ``the non-existence of holomorphic lines is an essentially useless qualitative statement without the quantitative information provided by the convergence of discs.'' It is common in complex geometry to circumvent this issue by studying sequences of disks. In particular, the following notion is quite useful:

\begin{dfn}\label{Ahlfors_current_dfn}
Let $(X,J,g)$ be an (almost) Hermitian manifold. A 2-current $T$ on $X$ is an \textit{Ahlfors current} if there is a sequence of $J$-holomorphic disks $D_i$ such that $\textnormal{length}(\partial D_i)/\textnormal{area}(D_i)$ tend to $0$ and the currents $D_i/\textnormal{area}(D_i)$ tend to $T$, weakly.
\end{dfn}

Recall that an \textit{almost Hermitian manifold} is an almost complex manifold $(X,J)$ with a $J$-invariant Riemannian metric $g$. Throughout, a $J$-holomorphic disk is taken to mean a $J$-holomorphic map $u: D \lra X$ from a closed disk $D \subset \C$ and the length and area are measured with respect to the pullback metric (equivalently, by integrating the relevant Jacobian).

\begin{rmk}
By the Ahlfors lemma \cite[Lemma~6.9]{Gro99} (see also \cites{Duv17}), every complex line in a compact manifold gives rise to an Ahlfors current. By Duval's quantitative Brody lemma~\cite[Théorème]{Duv08}, every Ahlfors current gives rise to a complex line. However, there is not a one-to-one correspondence between the two classes of objects~\cite{Huy21}.
\end{rmk}

The aim of this article is to understand the symplectic topology of Ahlfors currents. Throughout, we pursue an analogy with rational curves. In particular, existence theorems (via Gromov's compactness theorem~\cite{Gro85}), the bubbling phenomenon, and the study of moduli spaces are central to the study of rational curves. Our Theorems~\ref{existence_thm}, \ref{bubbling_thm}, and \ref{convexity_thm} address the corresponding points for Ahlfors currents.\\

We first address the existence of Ahlfors currents. Rational and elliptic curves are special cases of Ahlfors currents, and they are well-studied in symplectic geometry. However, more general Ahlfors currents are not amenable to the same techniques. It is difficult even to perturb a pre-existing complex line (cf.~\cite{Mos95}). When one perturbs a projective complex structure on a torus to a non-projective one, its subtori break up into complex lines that fill up the manifold. So, even local questions concerning complex lines can require global knowledge of the ambient manifold. This is part of the reason Theorem~\ref{Bangert_thm} is so remarkable. Our first result is a generalization thereof.

\begin{thm}\label{existence_thm}
Let $(X,\omega)$ be a closed connected symplectically aspherical manifold so that
\begin{enumerate}[label = (\alph*)]
%\item $\omega$ is aspherical,
\item\label{retraction_assumption} X retracts onto a closed Lagrangian submanifold $L_X$ and
\item\label{Dehn_assumption} the homological Dehn function $\delta_X$ of $\pi_1(X)$ grows at most quadratically.
\end{enumerate}
Then, there exists a nonconstant $J$-holomorphic map $u:\C \lra X \times T^2$ for every almost complex structure $J$ tamed by $\omega \oplus \omega_{std}$.
\end{thm}
\goodbreak

We recall that a retraction of a topological space $X$ onto a subspace $L_X$ is a continuous map from $X$ to $L_X$ which is the identity on $L_X$. In particular, a retraction is not necessarily a deformation retraction. The notion of homological Dehn function is recalled in Section~\ref{Filling_sec}. 

\begin{rmk}\label{nonsummable_rmk}
Both assumptions are stronger than needed. It suffices to assume that the lift of $L_X$ to the universal cover of $X$ is a $K$-Lipschitz retract for some $K \geq 1$. It also suffices to assume that $n/\delta_X(n)$ is not summable. This includes functions like $n^2\log n$, $n^2 \log \log n$, etc. It is not clear to the author how much further the assumptions can be weakened, as it can be quite difficult to show that a given almost complex manifold is hyperbolic.
\end{rmk}

\begin{crl}\label{surfaces_crl}
Let $(X,\omega)$ be a product of closed surfaces, equipped with the product symplectic form. If all genera are greater than $1$, then every almost complex structure $J$ tamed by $\omega$ is hyperbolic. If not, then no almost complex structure $J$ tamed by $\omega$ is hyperbolic.
\end{crl}

This answers a question of Ivashkovich and Rosay~\cite[Question~1]{Iva04} in full generality.

\begin{crl}\label{nonpositive_crl}
Let $(X,\omega)$ be a closed symplectic manifold with a metric $g$ of nonpositive curvature and a Lagrangian $L_X$ which is totally geodesic with respect to $g$. Then, no almost complex structure $J$ on $X \times T^2$ tamed by $\omega \oplus \omega_{std}$ is hyperbolic.
\end{crl}

Some results similar to Theorem~\ref{existence_thm} were obtained in \cite{Bio04} in the context of convex symplectic manifolds, under some geometric assumptions and restrictions on the almost complex structure, by completely different, Floer-theoretic methods. Ahlfors currents were constructed in a related context in~\cite{DuvGay14}.
The assumptions of Theorem~\ref{existence_thm} are topological and much weaker than the curvature assumption in Corollary~\ref{nonpositive_crl}. This is highlighted in Corollary~\ref{group_crl} below. For $n \geq 3$, the product space in this corollary has nontrivial $\pi_2$ \cite[Proposition~3.6]{DiC25} and thus admits no metric of nonpositive curvature by the Cartan-Hadamard theorem \cite[Chapter~IX, Theorem~3.8]{Lan99}. The $(g,n) = (1,1)$ case of Corollary~\ref{group_crl} is the $n=3$ case of Theorem~\ref{Bangert_thm}.

\begin{crl}\label{group_crl}
Let $(\textnormal{Sym}^n(\Sigma_g), \omega)$ be the $n$-fold symmetric product of a surface of genus $g$ with its natural symplectic structure. Suppose $g \geq 2n-1$. Then, no almost complex structure $J$ on $\textnormal{Sym}^n(\Sigma_g) \times \textnormal{Sym}^n(\Sigma_g) \times T^2$ tamed by $\omega \oplus -\omega \oplus \omega_{std}$ is hyperbolic. 
\end{crl}

The proofs of Theorem~\ref{existence_thm} and Corollaries~\ref{surfaces_crl}--\ref{group_crl} are contained in Section~\ref{Nonhyperbolicity_sec}. In contrast to~\cite{Ban98} and following the suggestion of~\cite[Section IV]{Gro00}, our proof proceeds via a study of Lagrangian submanifolds. It is conceptually quite simple. To construct a complex line, it suffices to find an Ahlfors current~\cite[Théorème]{Duv08}. Therefore, we just need to find holomorphic disks satisfying reverse isoperimetric inequalities. Duval~\cite{Duv16} showed that such inequalities follow from the monotonicity of area density for holomorphic curves. We introduce a relative notion of (quasi)minimizing surface and use it to prove a general monotonicity formula (Proposition~\ref{monotonicity_prp}) which applies at large scales, from which Theorem~\ref{existence_thm} follows quickly.\\

It seems that the simplest manifold to which Theorem~\ref{existence_thm} does not apply is the Kodaira-Thurston manifold, the quotient of the Heisenberg group times $\R$ by a lattice. It is aspherical, but it seems likely (via the reasoning in \cite[Section~6]{You13}) that its homological Dehn function grows at least cubically (cf. \cite[Chapter~8]{Eps92}). Furthermore, this should be true for any symplectic nilmanifold that is not a torus. Nevertheless, it seems possible that the proof technique (especially Proposition~\ref{monotonicity_prp}) applied to a cleverly chosen Lagrangian would recover the result.

\begin{ques}
Let $J$ be an almost complex structure on the Kodaira-Thurston manifold $X$ which is tamed by an invariant symplectic form. Is there a nonconstant $J$-holomorphic map $u: \C \lra X$? More generally, does this hold for every symplectic nilmanifold?
\end{ques}

Our second point of interest is bubbling. Classically, the Brody reparameterization lemma \cites{Bro78,Duv17} allows one to extract a complex line out of a sequence of holomorphic maps whose derivatives blow up. This implies that hyperbolic manifolds are rigid. For example, Demailly \cite[Theorem~2.1]{Dem97} has shown in the integrable setting that the area of a holomorphic curve in a hyperbolic manifold is bounded linearly by its genus (equivalently, its Euler characteristic).\\

Bubbling of rational curves is a special case of this reparameterization procedure. Just like in the Brody lemma, rational curve bubbles arise as a source of flexibility (for instance, noncompactness of spaces of holomorphic curves). However, unlike complex lines, rational curves can be controlled topologically.  
This allows one to control the flexibility that comes from bubbling even in cases where rational curves exist.
Our second theorem is a refinement of Demailly's result in which we produce an Ahlfors current. These can be controlled topologically, and therefore allow aspects of hyperbolicity to be used even in nonhyperbolic manifolds.

\begin{thm}\label{bubbling_thm}
Let $(X,J)$ be a closed almost complex manifold. Let $\Sigma_i$ be a sequence of closed $J$-holomorphic curves in $X$. If
\begin{equation}\label{bubbling_thm_eq}
\lim_{i \rightarrow \infty}\frac{\textnormal{genus}(\Sigma_i)}{\textnormal{area}(\Sigma_i)} = 0,
\end{equation}
then a subsequence of $\Sigma_i/\textnormal{area}(\Sigma_i)$ converges to an Ahlfors current.
\end{thm}

If $(X,J)$ is an almost complex manifold and $A \in H_2(X;\Z)$, we let $\ov{\mathfrak{M}}_g(A;J)$ denote the moduli space of stable $J$-holomorphic curves of genus $g$ representing $A$.

\begin{crl}\label{GW_crl}
Let $(X,J)$ be a closed almost complex manifold. If no Ahlfors current is null-homologous, then the moduli spaces $\ov{\mathfrak{M}}_g(A;J)$ are compact and Gromov-Witten type invariants enumerating $J$-holomorphic curves can be defined.
\end{crl}

To use Theorem~\ref{bubbling_thm}, one must be able to control the homology classes of Ahlfors currents. The following notion, due to \cite{Gro91}, is useful in this regard. 

\begin{dfn}\label{hyperbolic_form_dfn}
Let $X$ be a closed manifold. A closed $k$-form~$\al$ on~$X$ is \textit{hyperbolic} if there exists
a $(k\!-\!1)$-form~$\wt\be$ on the universal cover~$\wt{X}$ which is bounded
with respect to a metric pulled back from~$X$ and satisfies
$d\wt\be\!=\!\pi^*\al$, where \hbox{$\pi\!:\wt{X}\!\lra\!X$} is the covering map.
\end{dfn}

Whether a form is hyperbolic depends only on its cohomology class, so it is a topological condition. The following corollary allows one to extract geometric information from hyperbolic forms.

\begin{crl}\label{partial_hyperbolicity_crl}
Let $(X,J)$ be a closed almost complex manifold. Let $\alpha$ be a hyperbolic 2-form and $\omega$ be a symplectic form taming $J$. For every $\varepsilon > 0$, there exists $C > 0$ such that
\begin{equation}\label{partial_hyp_ineq}
\int_\Sigma \omega < C \cdot \textnormal{genus}(\Sigma)
\end{equation}
for every closed $J$-holomorphic curve $\Sigma$ in $X$ satisfying $\int_\Sigma \alpha > \varepsilon \int_\Sigma \omega.$
\end{crl}

We note that inequality ({\ref{partial_hyp_ineq}}) above is the reverse of the Castelnuovo bound in \cite[Theorem~1.6]{Doa21}, of interest due to its relation to the Gopakumar-Vafa finiteness conjecture.\\

Theorem~\ref{bubbling_thm} and Corollaries~\ref{GW_crl} and \ref{partial_hyperbolicity_crl} are proved in Section~\ref{bubbling_sec}. The theorem follows from an averaging trick. It works as follows. The genus of a surface is approximately its intrinsic (hyperbolic) area. If the ratio of its induced area and intrinsic area blows up, then the area distortion of the map is high, on average. Therefore, the average disk in the universal cover will also have high area distortion and thus large area. One must also control the boundary length. That this is also amenable to an averaging trick follows from the equidistribution of spheres in hyperbolic geometry.\\

Our final theorem concerns the moduli space of Ahlfors currents. Here, the situation is quite different from that of rational curves. Rational curves are solutions to a (nonlinear) elliptic PDE and therefore lie in finite-dimensional moduli spaces. Ahlfors currents, on the other hand, naturally lie in a linear space, and so different methods are appropriate. The next theorem shows that the space of Ahlfors currents is well-behaved in a linear sense.

\begin{thm}\label{convexity_thm}
Let $(X,J)$ be a closed connected almost complex manifold. The space of Ahlfors currents in $X$ is convex and weakly compact.
\end{thm}

It follows immediately from Theorem~\ref{convexity_thm} that the space of homology classes of Ahlfors currents and the space of Ahlfors currents lying in a given homology class are both also compact and convex. In light of the Krein-Milman theorem \cite[Theorem~3.23]{Rud91}, these sets can be understood in terms of their extreme points (cf. \cite[Section~1]{Phe01}). This was explored in the setting of rational curves in \cite{Rua93}. 
Example~\ref{extreme_points_eg} below demonstrates that, in simple cases, the space of extreme points contains the same information as the space of rational curves.

\begin{eg}\label{extreme_points_eg}
Let $X$ be a closed complex manifold which is a holomorphic $S^2$-bundle over a Kähler manifold, equipped with a Hermitian metric such that the fibers have unit area. Then, the space of extreme points of the cone of Ahlfors currents in the fiber class $F$ is the space $\ov{\mathfrak{M}}_0(F;J)$.
\end{eg}

\begin{ques}
What symplectic invariants can be extracted from the space of Ahlfors currents?
\end{ques}

The weak compactness part of Theorem~\ref{convexity_thm} is well-known; see~\cite[Lemma 26.14]{Sim83}, for example. We establish the convexity part of this theorem, as well as the claim in Example~\ref{extreme_points_eg}, in Section~\ref{convexity_sec}. The convexity claim is proved by a gluing argument. The main point is that having a nonempty boundary, but no boundary conditions, makes $J$-holomorphic disks flexible (cf. \cite{Suk12}).\\

\textbf{Acknowledgements.}  The author would like to thank G.~Antonelli, S.~Diverio, M.~Gromov, M.~Khuri, H.~B.~Lawson, E.~Murphy, H.~L.~Tanaka, D.~Varolin, R.~Young, and F.~Zheng for helpful conversations, A.~Zinger for his invaluable guidance during the writing of this manuscript, and the anonymous reviewers for helpful comments. The author also appreciates an anonymous comment pointing to some resemblance with the results of~\cite{Bio04}, which has been addressed.

\section{Existence of Ahlfors currents}\label{Nonhyperbolicity_sec}
This section is devoted to the proof of Theorem~\ref{existence_thm} and its corollaries. In order to construct a complex line, one must show the existence of holomorphic disks, as in Definition~\ref{Ahlfors_current_dfn}, satisfying arbitrarily strong length-area estimates. The existence of holomorphic disks (with Lagrangian boundary) is well-studied in symplectic topology. In our setting, none of the usual technical difficulties arise and disks can be shown to exist via a simple continuity argument. This is the content of Lemma~\ref{existence_of_disks_lmm}. Therefore, our primary focus is on the length-area estimates for $J$-holomorphic disks. Duval has shown in \cite{Duv16} (dramatically simplifying the proof in \cite{Gro14}) that such estimates follow from the monotonicity of the area density of a pseudoholomorphic curve. The larger the scale, the stronger the estimates obtained in \cite{Duv16}, but the argument in \cite{Duv16} works only on sufficiently small scales. As we need arbitrarily strong length-area estimates, we need a robust understanding of how the monotonicity lemma applies at large scales. This is the content of Section~\ref{Monotonicity_sec}. The necessary background is collected in Section~\ref{Filling_sec}.\\

The proof of Theorem~\ref{existence_thm} can be summarized as follows. The fundamental group of a compact manifold controls the filling function of its universal cover (Lemma~\ref{quasiisometry_invariance_lmm}). A filling function induces a relative filling function for well-behaved subsets (Lemma~\ref{relative_from_filling_lmm}). A relative filling function yields a monotonicity formula (Proposition~\ref{monotonicity_prp}). Upon shrinking the disk as in Lemma~\ref{isoperimetry_from_coarse_isoperimetry_lmm}, we obtain the desired length-area estimates, which become arbitrarily strong at large scales; see Remark~\ref{asymptotic_isoperimetry_rmk}. Finally, Lemma~\ref{existence_of_disks_lmm} implies the existence of disks to which these estimates can be applied.

\subsection{Filling functions}\label{Filling_sec}
This section recalls the necessary background conerning filling functions and homological Dehn functions. Throughout, we use (integer multiplicity) $n$-rectifiable currents with finite mass. For brevity, we will often refer to them as chains, currents, or cycles (when they are closed). Our primary reference is~\cite[Chapter~6]{Sim83}. We denote the mass of a current $C$ by $|C|$, its boundary by $\partial C$, and its restriction to a Borel subset $B$ by $C\res B$. We note that every pseudoholomorphic curve $C$ is 2-rectifiable and the area of $C$ equals its mass $|C|$. The following lemma describes how to decompose an $n$-current into a family of $(n-1)$-currents.

\begin{lmm}[{\cite[Lemma~28.5]{Sim83}}]\label{slicing_lmm}
Let $(X,g)$ be a Riemannian manifold, $C$ be an $n$-rectifiable current on $X$, and $f: X \lra \R$ be a $1$-Lipschitz function. Then, for almost every $r \in \R$, there is an $(n-1)$-rectifiable current $C_r(f)$ supported on $f^{-1}(r)$ such that for all $a,b \in \R$ with $a\leq b$,
\begin{equation}\label{slicing_formulas}
\begin{split}
\big|C \res f^{-1}([a,b])\big| &\geq \int_a^b |C_r(f)|dr \quad
\textnormal{and}\\ \partial\big(C \res f^{-1}(-\infty,r)\big) &= \partial C \res f^{-1}(-\infty,r) + C_r(f).
\end{split}
\end{equation}
\end{lmm}

Let $(X,g)$ be a Riemannian manifold, $d$ be the induced metric on $X$, and $L \subset X$ be nonempty. We define the \textit{distance function to $L$} by
$$d_L : X \lra {[}0,\infty), \quad d_L(p) := \inf_{x \in L} \{d(p,x)\}.$$
The triangle inequality implies that $d_L$ is $1$-Lipschitz.  This function will frequently take the place of $f$ in Lemma~\ref{slicing_lmm} in subsequent arguments. For $R \in (0,\infty{]}$, we denote the set $d_L^{-1}\big([0,R{)}\big)$ by $B_R(L)$.

\begin{dfn}\label{filling_function}
Let $(X,g)$ be a Riemannian manifold. A function $\varphi: {[}0,\infty) \lra {[}0,\infty)$ is a \textit{filling function} for $(X,g)$ if, for every $\varepsilon > 0$, every (not necessarily connected) 1-rectifiable cycle $\gamma$ in $X$ bounds a 2-rectifiable current~$C$ with $|C| < \varphi(|\gamma|) + \varepsilon$. The smallest such function is called the \textit{optimal filling function} for $X$.
\end{dfn}

We recall the following notion standard in geometric group theory. Let $f, g: {[}0,\infty) \lra {[}0,\infty)$ be functions. We say a $f$ dominates $g$ if there exists $C > 0$ such that
$$Cf(Cr + C) + C \geq g(r) \quad \forall r \in {[}0,\infty).$$
We say $f$ and $g$ are equivalent if $f$ dominates $g$ and $g$ dominates $f$.

\begin{figure}
\centering

\tikzset{every picture/.style={line width=0.75pt}} %set default line width to 0.75pt        

\begin{tikzpicture}[x=0.75pt,y=0.75pt,yscale=-.75,xscale=.75]
%uncomment if require: \path (0,300); %set diagram left start at 0, and has height of 300

%Shape: Ellipse [id:dp38830229464413923] 
\draw  [line width=2.25]  (90,65) .. controls (90,51.19) and (134.77,40) .. (190,40) .. controls (245.23,40) and (290,51.19) .. (290,65) .. controls (290,78.81) and (245.23,90) .. (190,90) .. controls (134.77,90) and (90,78.81) .. (90,65) -- cycle ;
%Straight Lines [id:da03311655557554949] 
\draw    (90,65) -- (190,255) ;
\draw [shift={(190,255)}, rotate = 62.24] [color={rgb, 255:red, 0; green, 0; blue, 0 }  ][fill={rgb, 255:red, 0; green, 0; blue, 0 }  ][line width=0.75]      (0, 0) circle [x radius= 3.35, y radius= 3.35]   ;
%Straight Lines [id:da9537398664320318] 
\draw    (290,65) -- (190,255) ;
%Shape: Ellipse [id:dp9731089110962197] 
\draw  [line width=2.25]  (452,66) .. controls (452,56.61) and (477.07,49) .. (508,49) .. controls (538.93,49) and (564,56.61) .. (564,66) .. controls (564,75.39) and (538.93,83) .. (508,83) .. controls (477.07,83) and (452,75.39) .. (452,66) -- cycle ;
%Straight Lines [id:da8685093434988118] 
\draw    (452,66) -- (449,227) ;
%Straight Lines [id:da9865386523918664] 
\draw    (564,66) -- (561,227) ;
%Shape: Arc [id:dp7305494109537591] 
\draw  [draw opacity=0][line width=0.75]  (561,227) .. controls (561,236.39) and (535.93,244) .. (505,244) .. controls (474.07,244) and (449,236.39) .. (449,227) -- (505,227) -- cycle ; \draw  [line width=0.75]  (561,227) .. controls (561,236.39) and (535.93,244) .. (505,244) .. controls (474.07,244) and (449,236.39) .. (449,227) ;  
%Shape: Arc [id:dp40049782744323825] 
\draw  [draw opacity=0][line width=0.75]  (561.5,171) .. controls (561.5,180.39) and (536.54,188) .. (505.75,188) .. controls (474.96,188) and (450,180.39) .. (450,171) -- (505.75,171) -- cycle ; \draw  [line width=0.75]  (561.5,171) .. controls (561.5,180.39) and (536.54,188) .. (505.75,188) .. controls (474.96,188) and (450,180.39) .. (450,171) ;  
%Shape: Arc [id:dp21640099240033628] 
\draw  [draw opacity=0][line width=0.75]  (562.5,117) .. controls (562.5,126.39) and (537.54,134) .. (506.75,134) .. controls (475.96,134) and (451,126.39) .. (451,117) -- (506.75,117) -- cycle ; \draw  [line width=0.75]  (562.5,117) .. controls (562.5,126.39) and (537.54,134) .. (506.75,134) .. controls (475.96,134) and (451,126.39) .. (451,117) ;  
%Straight Lines [id:da13609481648353938] 
\draw [line width=2.25]    (451,142) -- (445.9,142) -- (358,264) -- (563.1,264) -- (651,142) -- (563,141) ;
%Shape: Arc [id:dp669398052787392] 
\draw  [draw opacity=0][line width=0.75]  (260.24,120.94) .. controls (255.18,130.59) and (225.66,138) .. (190,138) .. controls (153.98,138) and (124.23,130.45) .. (119.62,120.65) -- (190,118) -- cycle ; \draw  [line width=0.75]  (260.24,120.94) .. controls (255.18,130.59) and (225.66,138) .. (190,138) .. controls (153.98,138) and (124.23,130.45) .. (119.62,120.65) ;  
%Shape: Arc [id:dp8045427064600397] 
\draw  [draw opacity=0][line width=0.75]  (233.05,173.11) .. controls (228.85,179.9) and (211.17,185) .. (190,185) .. controls (168.59,185) and (150.74,179.78) .. (146.81,172.88) -- (190,170) -- cycle ; \draw  [line width=0.75]  (233.05,173.11) .. controls (228.85,179.9) and (211.17,185) .. (190,185) .. controls (168.59,185) and (150.74,179.78) .. (146.81,172.88) ;  
%Shape: Arc [id:dp8710110490966521] 
\draw  [draw opacity=0][line width=0.75]  (213.32,210.92) .. controls (209.75,213.89) and (200.88,216) .. (190.5,216) .. controls (179.02,216) and (169.38,213.42) .. (166.73,209.94) -- (190.5,208) -- cycle ; \draw  [line width=0.75]  (213.32,210.92) .. controls (209.75,213.89) and (200.88,216) .. (190.5,216) .. controls (179.02,216) and (169.38,213.42) .. (166.73,209.94) ;  

% Text Node
\draw (69,54) node [anchor=north west][inner sep=0.75pt]   [align=left] {$\displaystyle \gamma $};
% Text Node
\draw (430,54) node [anchor=north west][inner sep=0.75pt]   [align=left] {$\displaystyle \gamma $};
% Text Node
\draw (165,247) node [anchor=north west][inner sep=0.75pt]   [align=left] {$\displaystyle p$};
% Text Node
\draw (593,155) node [anchor=north west][inner sep=0.75pt]   [align=left] {$\displaystyle L$};

\end{tikzpicture}
\caption{The geodesic cone over a curve centered at a point, as in Example~\ref{nonpositive_filling_function_eg}, is depicted on the left. The geodesic cylinder over a curve based on a totally geodesic submanifold, as in Example~\ref{totally_geodesic_eg}, is depicted on the right.}
\end{figure}
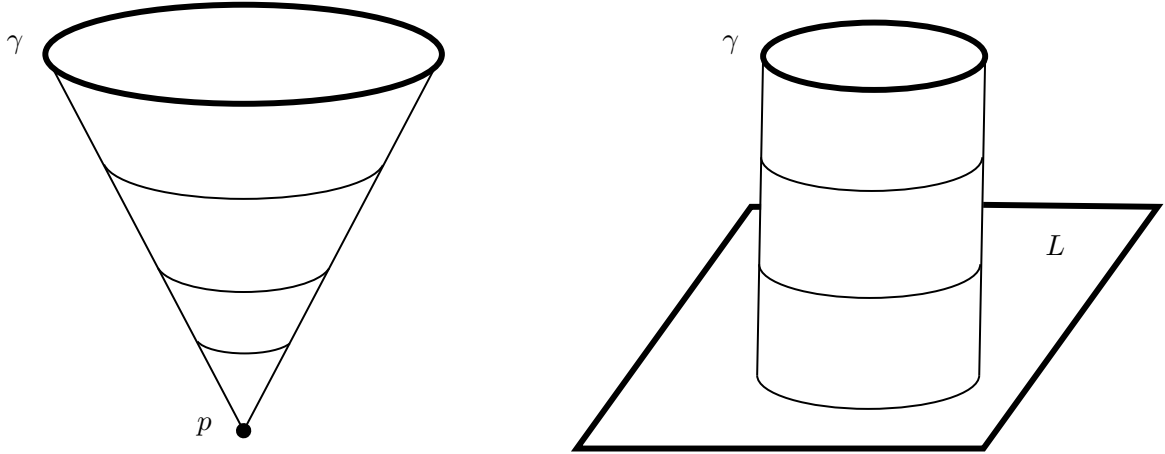

\begin{eg}\label{nonpositive_filling_function_eg}
Let $(X,g)$ be a simply connected Riemannian manifold with nonpositive curvature. Let $\gamma$ be a $1$-rectifiable cycle in $X$. Without loss of generality, assume $\gamma$ is connected. Then, it lies in a ball of diameter $|\gamma|$ around some point $p \in X$. Let $C$ be the geodesic cone over $\gamma$ centered at $p$. By \cite[Chapter~X, Theorems 2.4 and 2.5]{Lan99}, $C$ has area at most that of the corresponding Euclidean cone, which is bounded above by $|\gamma|^2$. Therefore, the filling function of $X$ is dominated by $r^2$.
\end{eg}

The following result shows that the equivalence class of the optimal filling function of the universal cover is, in fact, an invariant of the fundamental group. This equivalence class is called the \textit{homological Dehn function}.

\begin{lmm}[{\cite{brady2020homological}, cf.~\cite{burillo2001equivalence}}]\label{quasiisometry_invariance_lmm}
Let $(X,g_X)$ and $(Y,g_Y)$ be closed Riemannian manifolds. If their fundamental groups are isomorphic, the optimal filling functions for their universal covers are equivalent.
\end{lmm}

\begin{rmk}
In \cite{brady2020homological}, this result is stated and proved for fillings of Lipschitz chains, instead of rectifiable currents. By the Federer-Fleming deformation theorem \cite[Theorem~5.5]{Fed60}, either class of chains can be deformed into the other. This does not change the equivalence class of the function, so the difference is inconsequential.
\end{rmk}

\begin{crl}\label{free_abelian_crl}
Let $(X,g)$ be a closed Riemannian manifold with $\pi_1(X)$ free abelian. Then, the optimal filling function of the universal cover of $X$ is dominated by $r^2$.
\end{crl}

\begin{proof}
The fundamental group of $X$ is isomorphic to that of an $n$-torus. As tori have metrics of nonpositive curvature, the claim follows from Example~\ref{nonpositive_filling_function_eg}.
\end{proof}

\begin{lmm}\label{product_with_space_lmm}
Let $(X,g)$ be a Riemannian manifold and $\varphi$ be a filling function for $X$. Then, $\varphi(r) + r^2$ is a filling function for $X \times \R^n$.
\end{lmm}

\begin{proof}
Let $\gamma$ be a $1$-rectifiable cycle in $X \times \R^n$. Without loss of generality, we may assume $\gamma$ is connected. Then, the projection of $\gamma$ to $\R^n$ lies in a ball $B$ of diameter $|\gamma|$ with center $p$. Radially contracting $X \times B$ to $X \times p$ yields a $1$-rectifiable cycle $\gamma'$ in $X \times p$ and a $2$-rectifiable chain $C$ in $X \times B$ such that $|\gamma'| \leq |\gamma|$, $\partial C = \gamma - \gamma'$, and $|C| \leq |\gamma|^2$. By assumption, for any $\varepsilon > 0$, there exists a $2$-rectifiable $C'$ such that $\partial C' = \gamma'$ and $|C'| \leq \varphi(|\gamma'|) + \varepsilon$; the claim follows.
\end{proof}

\begin{lmm}[{cf.\cite{brady2020homological}}]\label{superadditivity_lmm}
Let $(X,g)$ be a closed Riemannian manifold and $(\wt{X},\wt{g})$ be its universal cover. If $\pi_1(X)$ is infinite, then the optimal filling function $\varphi$ for $\wt{X}$ is superadditive.
\end{lmm}

\begin{proof}
As $(\wt{X},\wt{g})$ covers a closed manifold, it has nonzero injectivity radius. Therefore, each connected 1-rectifiable cycle with sufficiently small length lies in a normal exponential neighborhood and can be filled by taking a cone. Applying this construction to each connected component of a 1-rectifiable cycle shows that $\lim_{r \rightarrow 0} \varphi(r) = 0$.\\

Let $\varepsilon > 0$. Take 1-rectifiable cycles $\gamma_1$ and $\gamma_2$ such that any rectifiable 2-current $C_i$ with boundary $\gamma_i$ satisfies $|C_i| \geq \varphi(|\gamma_i|) - \varepsilon$ for $i = 1,2$. Take $\delta > 0$ such that $\varphi(\delta) < \varepsilon$. By translating $\gamma_1$ by a deck transformation, we may assume $\gamma_1$ and $\gamma_2$ are a distance of at least $2(\varphi(|\gamma_1| + |\gamma_2|) + \varepsilon)/\delta$ apart. Let $C$ be a 2-rectifiable chain such that
$$\partial C = \gamma_1 + \gamma_2 \quad \textnormal{and} \quad |C| < \varphi \big(|\gamma_1| + |\gamma_2| \big) + \varepsilon.$$
Let $d_{\gamma_1}$ be the distance function from the support of $\gamma_1$. It is clearly 1-Lipschitz. Lemma~\ref{slicing_lmm} applied to $d_{\gamma_1}$ shows that there is an $r \in \R$ such that $C_r(d_{\gamma_1})$ is a 1-rectifiable cycle with length $|C_r(d_{\gamma_1})| < \delta$. It can therefore be filled in by a 2-rectifiable cycle $C'$ with area $|C'| < 2\varepsilon$. Let
$$C_1 = C \res d_{\gamma_1}^{-1}\big([0,r)\big) - C' \quad \textnormal{and} \quad C_2 = C \res d_{\gamma_1}^{-1}\big([r,\infty)\big) + C'.$$
It follows from (\ref{slicing_formulas}) that $\partial C_1 = \gamma_1$, $\partial C_2 = \gamma_2$, and
$$\varphi(|\gamma_1|) + \varphi(|\gamma_2|) - 2\varepsilon < |C_1| + |C_2| < \varphi(|\gamma_1| + |\gamma_2|) + 5\varepsilon.$$
As $\varepsilon$ is arbitrary, we are done.
\end{proof}

\usetikzlibrary{patterns}

% Pattern Info
 
\tikzset{
pattern size/.store in=\mcSize, 
pattern size = 5pt,
pattern thickness/.store in=\mcThickness, 
pattern thickness = 0.3pt,
pattern radius/.store in=\mcRadius, 
pattern radius = 1pt}
\makeatletter
\pgfutil@ifundefined{pgf@pattern@name@_g8zoek1yn}{
\pgfdeclarepatternformonly[\mcThickness,\mcSize]{_g8zoek1yn}
{\pgfqpoint{0pt}{0pt}}
{\pgfpoint{\mcSize+\mcThickness}{\mcSize+\mcThickness}}
{\pgfpoint{\mcSize}{\mcSize}}
{
\pgfsetcolor{\tikz@pattern@color}
\pgfsetlinewidth{\mcThickness}
\pgfpathmoveto{\pgfqpoint{0pt}{0pt}}
\pgfpathlineto{\pgfpoint{\mcSize+\mcThickness}{\mcSize+\mcThickness}}
\pgfusepath{stroke}
}}
\makeatother
\tikzset{every picture/.style={line width=0.75pt}} %set default line width to 0.75pt        

\begin{figure}
\centering
\begin{tikzpicture}[x=0.75pt,y=0.75pt,yscale=-.75,xscale=.75]
%uncomment if require: \path (0,504); %set diagram left start at 0, and has height of 504

%Shape: Ellipse [id:dp7981736545625591] 
\draw   (30,125) .. controls (30,83.58) and (38.51,50) .. (49,50) .. controls (59.49,50) and (68,83.58) .. (68,125) .. controls (68,166.42) and (59.49,200) .. (49,200) .. controls (38.51,200) and (30,166.42) .. (30,125) -- cycle ;
%Curve Lines [id:da06877110867299563] 
\draw    (49,50) .. controls (100,50) and (67,109.5) .. (130,110) .. controls (193,110.5) and (265,109.5) .. (325,110) ;
%Shape: Ellipse [id:dp8446753469794284] 
\draw  [pattern=_g8zoek1yn,pattern size=3.75pt,pattern thickness=0.75pt,pattern radius=0pt, pattern color={rgb, 255:red, 0; green, 0; blue, 0}] (320,125) .. controls (320,116.72) and (322.24,110) .. (325,110) .. controls (327.76,110) and (330,116.72) .. (330,125) .. controls (330,133.28) and (327.76,140) .. (325,140) .. controls (322.24,140) and (320,133.28) .. (320,125) -- cycle ;
%Curve Lines [id:da8717729482438394] 
\draw    (50,200) .. controls (101,200) and (68,140.5) .. (131,140) .. controls (194,139.5) and (266,140.5) .. (326,140) ;
%Shape: Ellipse [id:dp784439845183921] 
\draw   (620,125) .. controls (620,83.58) and (611.49,50) .. (601,50) .. controls (590.51,50) and (582,83.58) .. (582,125) .. controls (582,166.42) and (590.51,200) .. (601,200) .. controls (611.49,200) and (620,166.42) .. (620,125) -- cycle ;
%Curve Lines [id:da9405051920651191] 
\draw    (601,50) .. controls (550,50) and (583,109.5) .. (520,110) .. controls (457,110.5) and (385,109.5) .. (325,110) ;
%Curve Lines [id:da7960631097416258] 
\draw    (600,200) .. controls (549,200) and (582,140.5) .. (519,140) .. controls (456,139.5) and (384,140.5) .. (324,140) ;
%Straight Lines [id:da8940434376013964] 
\draw    (357,100) -- (335.41,121.59) ;
\draw [shift={(334,123)}, rotate = 315] [color={rgb, 255:red, 0; green, 0; blue, 0 }  ][line width=0.75]    (10.93,-3.29) .. controls (6.95,-1.4) and (3.31,-0.3) .. (0,0) .. controls (3.31,0.3) and (6.95,1.4) .. (10.93,3.29)   ;

% Text Node
\draw (41,202) node [anchor=north west][inner sep=0.75pt]   [align=left] {$\displaystyle \gamma _{1}$};
% Text Node
\draw (590,202) node [anchor=north west][inner sep=0.75pt]   [align=left] {$\displaystyle \gamma _{2}$};
% Text Node
\draw (296,148) node [anchor=north west][inner sep=0.75pt]   [align=left] {$\displaystyle C_{r}( d_{\gamma _{1}})$};
% Text Node
\draw (351,80) node [anchor=north west][inner sep=0.75pt]   [align=left] {$\displaystyle C'$};

\end{tikzpicture}

\caption{The chain $C$ filling $\gamma_1 + \gamma_2$ in the proof of Lemma~\ref{superadditivity_lmm}}
\end{figure}
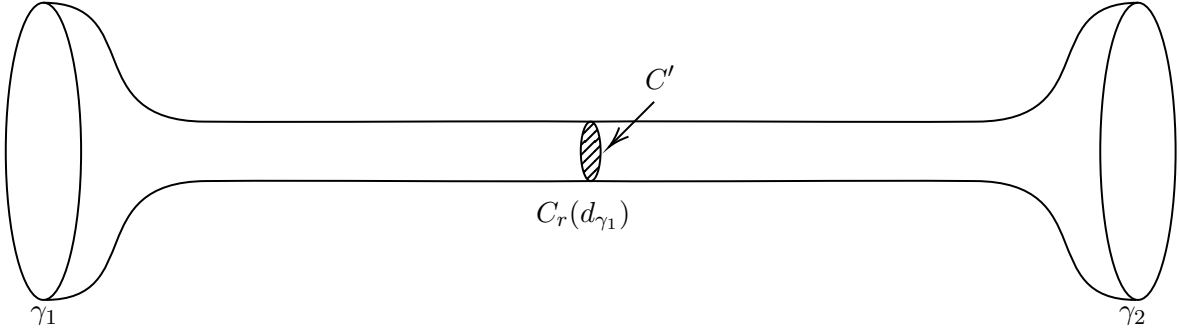

\begin{dfn}\label{relative_filling_function}
Let $(X,g)$ be a Riemannian manifold, $L \subset X$, and $R \in (0,\infty{]}$. A function $$\varphi_L:{[}0,\infty) \lra {[}0,\infty)$$ is a \textit{filling function for $(X,g)$ $R$-relative to $L$} if there is a retraction
\begin{equation}\label{retraction_eqn}
\eta: H_1\big(B_R(L);\Z\big) \lra H_1(L;\Z)
\end{equation}
and, for every $\varepsilon > 0$, $r \in (0,R)$, and rectifiable cycle $\gamma$ in $B_r(L)$, there exist a cycle $\gamma'$ in $L$ and a 2-rectifiable current $C$ in $X$ such that $[\gamma'] = \eta\big([\gamma]\big)$, $\partial C = \gamma - \gamma'$, and $|C| < \varphi_L(r)|\gamma| + \varepsilon$.
\end{dfn}

By $\eta$ in (\ref{retraction_eqn}) being a \textit{retraction}, we mean that $\eta$ is a homomorphism such that the composition
$$\eta \circ \iota_* : H_1(L;\Z) \lra H_1\big(B_R(L); \Z\big) \lra H_1(L;\Z)$$
of $\eta$ with the homomorphism $\iota_*$ induced by the inclusion $\iota:L \lra B_R(L)$ is the identity on~$H_1(L;\Z)$. We do not assume that $L$ is compact, that $\eta$ is induced by a continuous retraction $B_R(L) \lra L$, or even that such a retraction exists. Furthermore, $\eta$ may depend on $R$.

\begin{rmk}\label{asymptotics_rmk}
We are primarily interested in the case where a given function $\varphi$ is a filling function $R_i$-relative to $L_i$, for a sequence $R_i$ tending to infinity and a sequence of subspaces $L_i$. This is why we have taken the domain of an $R$-relative filling function to be ${[}0,\infty)$ instead of just ${[}0,R)$. It is also for this reason that we consider asymptotics of relative filling functions, despite the fact that the behavior of an $R$-relative filling function is meaningful only on ${[}0,R)$.
\end{rmk}

\begin{rmk}\label{homological_assumption_rel_filling_rmk}
It is important that the current $C$ is not required to lie in $B_r(L)$. Otherwise, it would be much more difficult to find relative filling functions. This is the reason for the homological assumptions in the definition.
\end{rmk}

\begin{rmk}
We note that Definition~\ref{filling_function} is not a special case of Definition~\ref{relative_filling_function}. The dependence on $|\gamma|$ is nonlinear in the former and linear in the latter. This is an important point in Propositions~\ref{monotonicity_prp} and \ref{lower_area_bound_prp}.
\end{rmk}

\begin{eg}\label{totally_geodesic_eg}
Let $(X,g)$ be a simply connected Riemannian manifold with nonpositive curvature and $L \subset X$ be a totally geodesic submanifold. Let $r > 0$ and $\gamma$ be a $1$-rectifiable cycle in $B_r(L)$. Let $C$ be the geodesic cylinder over $\gamma$ based on $L$. By \cite[Chapter~X, Theorems 2.4 and 2.5]{Lan99}, $C$ has area at most that of the corresponding Euclidean cylinder, which is bounded above by $r|\gamma|$. Therefore, the identity is a filling function $\infty$-relative to $L$.
\end{eg}

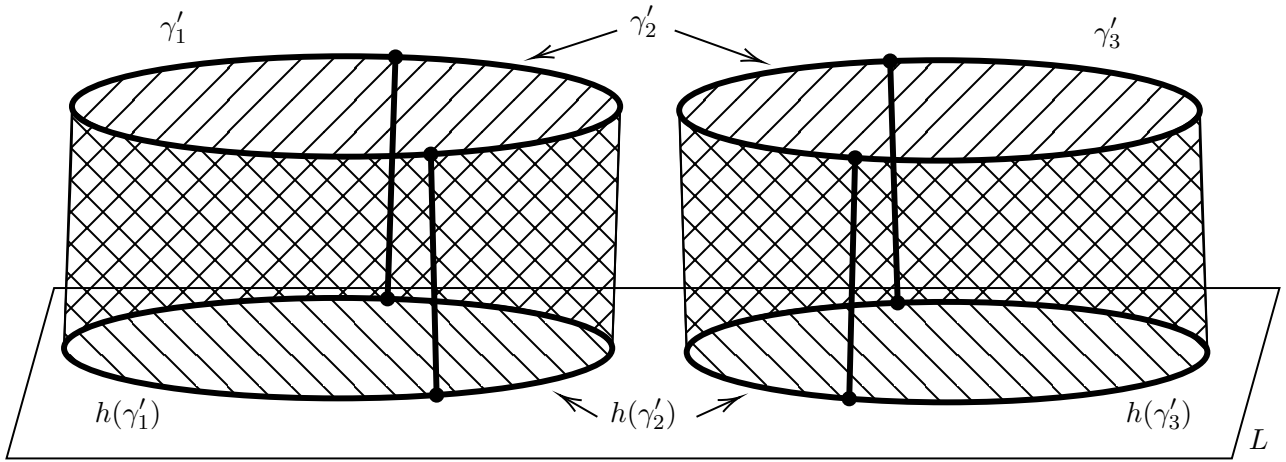
\begin{figure}
\centering
% Pattern Info

% Pattern Info
 
\tikzset{
pattern size/.store in=\mcSize, 
pattern size = 5pt,
pattern thickness/.store in=\mcThickness, 
pattern thickness = 0.3pt,
pattern radius/.store in=\mcRadius, 
pattern radius = 1pt}
\makeatletter
\pgfutil@ifundefined{pgf@pattern@name@_pl86b5ijy}{
\pgfdeclarepatternformonly[\mcThickness,\mcSize]{_pl86b5ijy}
{\pgfqpoint{0pt}{0pt}}
{\pgfpoint{\mcSize+\mcThickness}{\mcSize+\mcThickness}}
{\pgfpoint{\mcSize}{\mcSize}}
{
\pgfsetcolor{\tikz@pattern@color}
\pgfsetlinewidth{\mcThickness}
\pgfpathmoveto{\pgfqpoint{0pt}{0pt}}
\pgfpathlineto{\pgfpoint{\mcSize+\mcThickness}{\mcSize+\mcThickness}}
\pgfusepath{stroke}
}}
\makeatother

% Pattern Info
 
\tikzset{
pattern size/.store in=\mcSize, 
pattern size = 5pt,
pattern thickness/.store in=\mcThickness, 
pattern thickness = 0.3pt,
pattern radius/.store in=\mcRadius, 
pattern radius = 1pt}
\makeatletter
\pgfutil@ifundefined{pgf@pattern@name@_57qncubsb}{
\pgfdeclarepatternformonly[\mcThickness,\mcSize]{_57qncubsb}
{\pgfqpoint{0pt}{-\mcThickness}}
{\pgfpoint{\mcSize}{\mcSize}}
{\pgfpoint{\mcSize}{\mcSize}}
{
\pgfsetcolor{\tikz@pattern@color}
\pgfsetlinewidth{\mcThickness}
\pgfpathmoveto{\pgfqpoint{0pt}{\mcSize}}
\pgfpathlineto{\pgfpoint{\mcSize+\mcThickness}{-\mcThickness}}
\pgfusepath{stroke}
}}
\makeatother

% Pattern Info
 
\tikzset{
pattern size/.store in=\mcSize, 
pattern size = 5pt,
pattern thickness/.store in=\mcThickness, 
pattern thickness = 0.3pt,
pattern radius/.store in=\mcRadius, 
pattern radius = 1pt}
\makeatletter
\pgfutil@ifundefined{pgf@pattern@name@_ouze8wquz}{
\pgfdeclarepatternformonly[\mcThickness,\mcSize]{_ouze8wquz}
{\pgfqpoint{0pt}{0pt}}
{\pgfpoint{\mcSize}{\mcSize}}
{\pgfpoint{\mcSize}{\mcSize}}
{
\pgfsetcolor{\tikz@pattern@color}
\pgfsetlinewidth{\mcThickness}
\pgfpathmoveto{\pgfqpoint{0pt}{\mcSize}}
\pgfpathlineto{\pgfpoint{\mcSize+\mcThickness}{-\mcThickness}}
\pgfpathmoveto{\pgfqpoint{0pt}{0pt}}
\pgfpathlineto{\pgfpoint{\mcSize+\mcThickness}{\mcSize+\mcThickness}}
\pgfusepath{stroke}
}}
\makeatother

% Pattern Info
 
\tikzset{
pattern size/.store in=\mcSize, 
pattern size = 5pt,
pattern thickness/.store in=\mcThickness, 
pattern thickness = 0.3pt,
pattern radius/.store in=\mcRadius, 
pattern radius = 1pt}
\makeatletter
\pgfutil@ifundefined{pgf@pattern@name@_6fz184bt7}{
\pgfdeclarepatternformonly[\mcThickness,\mcSize]{_6fz184bt7}
{\pgfqpoint{0pt}{0pt}}
{\pgfpoint{\mcSize+\mcThickness}{\mcSize+\mcThickness}}
{\pgfpoint{\mcSize}{\mcSize}}
{
\pgfsetcolor{\tikz@pattern@color}
\pgfsetlinewidth{\mcThickness}
\pgfpathmoveto{\pgfqpoint{0pt}{0pt}}
\pgfpathlineto{\pgfpoint{\mcSize+\mcThickness}{\mcSize+\mcThickness}}
\pgfusepath{stroke}
}}
\makeatother

% Pattern Info
 
\tikzset{
pattern size/.store in=\mcSize, 
pattern size = 5pt,
pattern thickness/.store in=\mcThickness, 
pattern thickness = 0.3pt,
pattern radius/.store in=\mcRadius, 
pattern radius = 1pt}
\makeatletter
\pgfutil@ifundefined{pgf@pattern@name@_ix2v6aqg5}{
\pgfdeclarepatternformonly[\mcThickness,\mcSize]{_ix2v6aqg5}
{\pgfqpoint{0pt}{-\mcThickness}}
{\pgfpoint{\mcSize}{\mcSize}}
{\pgfpoint{\mcSize}{\mcSize}}
{
\pgfsetcolor{\tikz@pattern@color}
\pgfsetlinewidth{\mcThickness}
\pgfpathmoveto{\pgfqpoint{0pt}{\mcSize}}
\pgfpathlineto{\pgfpoint{\mcSize+\mcThickness}{-\mcThickness}}
\pgfusepath{stroke}
}}
\makeatother

% Pattern Info
 
\tikzset{
pattern size/.store in=\mcSize, 
pattern size = 5pt,
pattern thickness/.store in=\mcThickness, 
pattern thickness = 0.3pt,
pattern radius/.store in=\mcRadius, 
pattern radius = 1pt}
\makeatletter
\pgfutil@ifundefined{pgf@pattern@name@_n3zlzheya}{
\pgfdeclarepatternformonly[\mcThickness,\mcSize]{_n3zlzheya}
{\pgfqpoint{0pt}{0pt}}
{\pgfpoint{\mcSize}{\mcSize}}
{\pgfpoint{\mcSize}{\mcSize}}
{
\pgfsetcolor{\tikz@pattern@color}
\pgfsetlinewidth{\mcThickness}
\pgfpathmoveto{\pgfqpoint{0pt}{\mcSize}}
\pgfpathlineto{\pgfpoint{\mcSize+\mcThickness}{-\mcThickness}}
\pgfpathmoveto{\pgfqpoint{0pt}{0pt}}
\pgfpathlineto{\pgfpoint{\mcSize+\mcThickness}{\mcSize+\mcThickness}}
\pgfusepath{stroke}
}}
\makeatother
\tikzset{every picture/.style={line width=0.75pt}} %set default line width to 0.75pt        

\begin{tikzpicture}[x=0.75pt,y=0.75pt,yscale=-.75,xscale=.75]
%uncomment if require: \path (0,300); %set diagram left start at 0, and has height of 300

%Shape: Ellipse [id:dp7026076551702687] 
\draw  [pattern=_pl86b5ijy,pattern size=12pt,pattern thickness=0.75pt,pattern radius=0pt, pattern color={rgb, 255:red, 0; green, 0; blue, 0}][line width=2.25]  (41,105.34) .. controls (41,91.35) and (102.9,80) .. (179.25,80) .. controls (255.6,80) and (317.5,91.35) .. (317.5,105.34) .. controls (317.5,119.34) and (255.6,130.69) .. (179.25,130.69) .. controls (102.9,130.69) and (41,119.34) .. (41,105.34) -- cycle ;
%Shape: Arc [id:dp8110438077762645] 
\draw  [draw opacity=0] (222.03,129.45) .. controls (208.56,130.26) and (194.18,130.69) .. (179.25,130.69) .. controls (102.9,130.69) and (41,119.34) .. (41,105.34) .. controls (41,91.35) and (102.9,80) .. (179.25,80) .. controls (187.76,80) and (196.1,80.14) .. (204.18,80.41) -- (179.25,105.34) -- cycle ; \draw    (222.03,129.45) .. controls (208.56,130.26) and (194.18,130.69) .. (179.25,130.69) .. controls (102.9,130.69) and (41,119.34) .. (41,105.34) .. controls (41,91.35) and (102.9,80) .. (179.25,80) .. controls (187.76,80) and (196.1,80.14) .. (204.18,80.41) ; \draw [shift={(204.18,80.41)}, rotate = 1.45] [color={rgb, 255:red, 0; green, 0; blue, 0 }  ][fill={rgb, 255:red, 0; green, 0; blue, 0 }  ][line width=0.75]      (0, 0) circle [x radius= 3.35, y radius= 3.35]   ; \draw [shift={(222.03,129.45)}, rotate = 177.12] [color={rgb, 255:red, 0; green, 0; blue, 0 }  ][fill={rgb, 255:red, 0; green, 0; blue, 0 }  ][line width=0.75]      (0, 0) circle [x radius= 3.35, y radius= 3.35]   ;
%Straight Lines [id:da7244628520327137] 
\draw    (316,67) -- (274.9,80.38) ;
\draw [shift={(273,81)}, rotate = 341.97] [color={rgb, 255:red, 0; green, 0; blue, 0 }  ][line width=0.75]    (10.93,-3.29) .. controls (6.95,-1.4) and (3.31,-0.3) .. (0,0) .. controls (3.31,0.3) and (6.95,1.4) .. (10.93,3.29)   ;
%Shape: Ellipse [id:dp482100964224625] 
\draw  [pattern=_57qncubsb,pattern size=12pt,pattern thickness=0.75pt,pattern radius=0pt, pattern color={rgb, 255:red, 0; green, 0; blue, 0}][line width=2.25]  (37,227.34) .. controls (37,213.35) and (98.9,202) .. (175.25,202) .. controls (251.6,202) and (313.5,213.35) .. (313.5,227.34) .. controls (313.5,241.34) and (251.6,252.69) .. (175.25,252.69) .. controls (98.9,252.69) and (37,241.34) .. (37,227.34) -- cycle ;
%Shape: Arc [id:dp6379656801345812] 
\draw  [draw opacity=0] (224.92,251) .. controls (209.51,252.09) and (192.76,252.69) .. (175.25,252.69) .. controls (98.9,252.69) and (37,241.34) .. (37,227.34) .. controls (37,213.35) and (98.9,202) .. (175.25,202) .. controls (183.76,202) and (192.1,202.14) .. (200.18,202.41) -- (175.25,227.34) -- cycle ; \draw    (224.92,251) .. controls (209.51,252.09) and (192.76,252.69) .. (175.25,252.69) .. controls (98.9,252.69) and (37,241.34) .. (37,227.34) .. controls (37,213.35) and (98.9,202) .. (175.25,202) .. controls (183.76,202) and (192.1,202.14) .. (200.18,202.41) ; \draw [shift={(200.18,202.41)}, rotate = 1.45] [color={rgb, 255:red, 0; green, 0; blue, 0 }  ][fill={rgb, 255:red, 0; green, 0; blue, 0 }  ][line width=0.75]      (0, 0) circle [x radius= 3.35, y radius= 3.35]   ; \draw [shift={(224.92,251)}, rotate = 176.52] [color={rgb, 255:red, 0; green, 0; blue, 0 }  ][fill={rgb, 255:red, 0; green, 0; blue, 0 }  ][line width=0.75]      (0, 0) circle [x radius= 3.35, y radius= 3.35]   ;
%Straight Lines [id:da09341988845490268] 
\draw    (41,105.34) -- (37,227.34) ;
%Straight Lines [id:da20917013238525595] 
\draw    (317.5,105.34) -- (313.5,227.34) ;
%Shape: Polygon Curved [id:ds2704354957403792] 
\draw  [pattern=_ouze8wquz,pattern size=12pt,pattern thickness=0.75pt,pattern radius=0pt, pattern color={rgb, 255:red, 0; green, 0; blue, 0}] (41,105.34) .. controls (40,97) and (29,131) .. (179.25,130.69) .. controls (329.5,130.38) and (318,101) .. (317.5,105.34) .. controls (317,109.69) and (314,219.69) .. (313.5,227.34) .. controls (313,235) and (328,201) .. (175.25,202) .. controls (22.5,203) and (37,231.84) .. (37,227.34) .. controls (37,222.85) and (42,113.69) .. (41,105.34) -- cycle ;
%Straight Lines [id:da14943758970097343] 
\draw [line width=2.25]    (204.18,80.41) -- (200.18,202.41) ;
%Straight Lines [id:da8553317968343948] 
\draw [line width=2.25]    (222.03,129.45) -- (224.92,251) ;
%Shape: Ellipse [id:dp31872747519805333] 
\draw  [pattern=_6fz184bt7,pattern size=12pt,pattern thickness=0.75pt,pattern radius=0pt, pattern color={rgb, 255:red, 0; green, 0; blue, 0}][line width=2.25]  (609.85,107.34) .. controls (609.85,93.35) and (551.07,82) .. (478.55,82) .. controls (406.04,82) and (347.25,93.35) .. (347.25,107.34) .. controls (347.25,121.34) and (406.04,132.69) .. (478.55,132.69) .. controls (551.07,132.69) and (609.85,121.34) .. (609.85,107.34) -- cycle ;
%Shape: Arc [id:dp401416639160902] 
\draw  [draw opacity=0] (435.99,131.33) .. controls (449.33,132.21) and (463.65,132.69) .. (478.55,132.69) .. controls (551.07,132.69) and (609.85,121.34) .. (609.85,107.34) .. controls (609.85,93.35) and (551.07,82) .. (478.55,82) .. controls (470.04,82) and (461.72,82.16) .. (453.66,82.45) -- (478.55,107.34) -- cycle ; \draw    (435.99,131.33) .. controls (449.33,132.21) and (463.65,132.69) .. (478.55,132.69) .. controls (551.07,132.69) and (609.85,121.34) .. (609.85,107.34) .. controls (609.85,93.35) and (551.07,82) .. (478.55,82) .. controls (470.04,82) and (461.72,82.16) .. (453.66,82.45) ; \draw [shift={(453.66,82.45)}, rotate = 178.39] [color={rgb, 255:red, 0; green, 0; blue, 0 }  ][fill={rgb, 255:red, 0; green, 0; blue, 0 }  ][line width=0.75]      (0, 0) circle [x radius= 3.35, y radius= 3.35]   ; \draw [shift={(435.99,131.33)}, rotate = 3.19] [color={rgb, 255:red, 0; green, 0; blue, 0 }  ][fill={rgb, 255:red, 0; green, 0; blue, 0 }  ][line width=0.75]      (0, 0) circle [x radius= 3.35, y radius= 3.35]   ;
%Straight Lines [id:da5134046704840283] 
\draw    (345,67) -- (387.63,82.32) ;
\draw [shift={(389.51,83)}, rotate = 199.77] [color={rgb, 255:red, 0; green, 0; blue, 0 }  ][line width=0.75]    (10.93,-3.29) .. controls (6.95,-1.4) and (3.31,-0.3) .. (0,0) .. controls (3.31,0.3) and (6.95,1.4) .. (10.93,3.29)   ;
%Shape: Ellipse [id:dp246186241243038] 
\draw  [pattern=_ix2v6aqg5,pattern size=12pt,pattern thickness=0.75pt,pattern radius=0pt, pattern color={rgb, 255:red, 0; green, 0; blue, 0}][line width=2.25]  (613.65,229.34) .. controls (613.65,215.35) and (554.86,204) .. (482.35,204) .. controls (409.83,204) and (351.05,215.35) .. (351.05,229.34) .. controls (351.05,243.34) and (409.83,254.69) .. (482.35,254.69) .. controls (554.86,254.69) and (613.65,243.34) .. (613.65,229.34) -- cycle ;
%Shape: Arc [id:dp455300305458217] 
\draw  [draw opacity=0] (433.02,252.84) .. controls (448.25,254.03) and (464.9,254.69) .. (482.35,254.69) .. controls (554.86,254.69) and (613.65,243.34) .. (613.65,229.34) .. controls (613.65,215.35) and (554.86,204) .. (482.35,204) .. controls (473.84,204) and (465.52,204.16) .. (457.46,204.45) -- (482.35,229.34) -- cycle ; \draw    (433.02,252.84) .. controls (448.25,254.03) and (464.9,254.69) .. (482.35,254.69) .. controls (554.86,254.69) and (613.65,243.34) .. (613.65,229.34) .. controls (613.65,215.35) and (554.86,204) .. (482.35,204) .. controls (473.84,204) and (465.52,204.16) .. (457.46,204.45) ; \draw [shift={(457.46,204.45)}, rotate = 178.39] [color={rgb, 255:red, 0; green, 0; blue, 0 }  ][fill={rgb, 255:red, 0; green, 0; blue, 0 }  ][line width=0.75]      (0, 0) circle [x radius= 3.35, y radius= 3.35]   ; \draw [shift={(433.02,252.84)}, rotate = 3.86] [color={rgb, 255:red, 0; green, 0; blue, 0 }  ][fill={rgb, 255:red, 0; green, 0; blue, 0 }  ][line width=0.75]      (0, 0) circle [x radius= 3.35, y radius= 3.35]   ;
%Straight Lines [id:da05703655121277129] 
\draw    (609.85,107.34) -- (613.65,229.34) ;
%Straight Lines [id:da08050244233992587] 
\draw    (347.25,107.34) -- (351.05,229.34) ;
%Shape: Polygon Curved [id:ds6301501854141809] 
\draw  [pattern=_n3zlzheya,pattern size=12pt,pattern thickness=0.75pt,pattern radius=0pt, pattern color={rgb, 255:red, 0; green, 0; blue, 0}] (609.85,107.34) .. controls (610.8,99) and (621.25,133) .. (478.55,132.69) .. controls (335.85,132.38) and (346.78,103) .. (347.25,107.34) .. controls (347.73,111.69) and (350.57,221.69) .. (351.05,229.34) .. controls (351.52,237) and (337.28,203) .. (482.35,204) .. controls (627.42,205) and (613.65,233.84) .. (613.65,229.34) .. controls (613.65,224.85) and (608.9,115.69) .. (609.85,107.34) -- cycle ;
%Straight Lines [id:da18886800118000657] 
\draw [line width=2.25]    (453.66,82.45) -- (457.46,204.45) ;
%Straight Lines [id:da9768455527280853] 
\draw [line width=2.25]    (435.99,131.33) -- (433.02,252.84) ;
%Shape: Parallelogram [id:dp47773705779394393] 
\draw   (32,197) -- (650,197) -- (626,283) -- (8,283) -- cycle ;
%Straight Lines [id:da9908308391218235] 
\draw    (306,262) -- (288.73,252) ;
\draw [shift={(287,251)}, rotate = 30.07] [color={rgb, 255:red, 0; green, 0; blue, 0 }  ][line width=0.75]    (10.93,-3.29) .. controls (6.95,-1.4) and (3.31,-0.3) .. (0,0) .. controls (3.31,0.3) and (6.95,1.4) .. (10.93,3.29)   ;
%Straight Lines [id:da885687743679489] 
\draw    (356,262) -- (376.18,252.83) ;
\draw [shift={(378,252)}, rotate = 155.56] [color={rgb, 255:red, 0; green, 0; blue, 0 }  ][line width=0.75]    (10.93,-3.29) .. controls (6.95,-1.4) and (3.31,-0.3) .. (0,0) .. controls (3.31,0.3) and (6.95,1.4) .. (10.93,3.29)   ;

% Text Node
\draw (84,56) node [anchor=north west][inner sep=0.75pt]   [align=left] {$\displaystyle \gamma '_{1}$};
% Text Node
\draw (321,53) node [anchor=north west][inner sep=0.75pt]   [align=left] {$\displaystyle \gamma '_{2}$};
% Text Node
\draw (555,59) node [anchor=north west][inner sep=0.75pt]   [align=left] {$\displaystyle \gamma '_{3}$};
% Text Node
\draw (51,251) node [anchor=north west][inner sep=0.75pt]   [align=left] {$\displaystyle h( \gamma '_{1})$};
% Text Node
\draw (311,251) node [anchor=north west][inner sep=0.75pt]   [align=left] {$\displaystyle h( \gamma '_{2})$};
% Text Node
\draw (571,251) node [anchor=north west][inner sep=0.75pt]   [align=left] {$\displaystyle h( \gamma '_{3})$};
% Text Node
\draw (633,267) node [anchor=north west][inner sep=0.75pt]   [align=left] {$\displaystyle L$};

\end{tikzpicture}
\caption{The shaded region is the chain $C$ constructed in the proof of Lemma~\ref{relative_from_filling_lmm}.}
\end{figure}

\begin{lmm}\label{relative_from_filling_lmm}
Let $(X,g)$ be a Riemannian manifold, $L \subset X$, and $\varphi$ be a superadditive filling function for $X$. If $K \geq 1$, $R>0$, and there is a $K$-Lipschitz retraction $h: (B_R(L),g) \lra (L,g)$, then $\varphi_L(r) := 2\varphi(\relbound Kr)/r$ is a filling function $R$-relative to $L$.
\end{lmm}

\begin{proof}
Let $r \in (0,R)$, $\gamma$ be a 1-rectifiable cycle in $B_r(L)$, and $\varepsilon > 0$. Suppose $|\gamma| \leq r$. As $h$ is $K$-Lipschitz, $|h(\gamma)| \leq K|\gamma|$. Therefore, $\gamma - h(\gamma)$ bounds a chain $C$ with mass at most $\varphi(2K|\gamma|) + \varepsilon$.
By the superadditivity (and therefore monotonicity) of $\varphi$,
$$\varphi(\relbound Kr) \geq \varphi\big(\relbound K \lfloor r/|\gamma| \rfloor |\gamma|\big) \geq
 \lfloor r/|\gamma| \rfloor \varphi\big(\relbound K|\gamma|\big) \geq r\varphi(2K|\gamma|)/(2|\gamma|),$$
where $\lfloor r/|\gamma| \rfloor$ denotes the greatest integer less than $r/|\gamma|$. Thus, $|C| \leq \varphi_L(r)|\gamma| + \varepsilon$, as needed.\\

Now suppose $|\gamma|>r$ and $\varepsilon$ is smaller than $r$ and $R - r$. There is a 1-rectifiable cycle $\gamma'$ parameterized by a finite union of circles such that the support of $\gamma'$ is $\varepsilon$-close to that of $\gamma$, $|\gamma'| < |\gamma| + \varepsilon$, and there is a chain of area at most $\varepsilon$ filling $\gamma - \gamma'$. This follows quickly from the definition of rectifiable curve or more generally from \cite[Corollary~8.23]{Fed60}.\\

Let $\lceil |\gamma'|/ r \rceil$ denote the least integer greater than $|\gamma'|/ r$. Subdivide $\gamma'$ into $\lceil |\gamma'|/ r \rceil$ pieces $\gamma'_i$, each with $|\gamma'_i| \leq r$ and at most $4$ boundary points (counted with multiplicity); a naive, greedy subdivision works. Each boundary point $p$ is at a distance at most $r + \varepsilon$ from a point $p' \in L$. Therefore, $h(p)$ is at a distance at most $K(r + \varepsilon)$ from $p'$ and at most $(K+1)(r + \varepsilon)$ from $p$. For each $\gamma'_i$, connect each boundary point to its image under $h$ by a path of length at most $(K+1)(r + \varepsilon)$. These paths, $\gamma'_i$, and $h(\gamma'_i)$ together form a 1-rectifiable cycle. It has length at most $20Kr$ and therefore can be filled by a chain of area at most $\varphi(20Kr) + \varepsilon$. The sum of the chains corresponding to each $\gamma'_i$ yields a chain $C$ of area at most $(\varphi(20Kr) + \varepsilon)\lceil |\gamma'| / r \rceil$ filling $\gamma' - h(\gamma')$. Since
$$\lceil |\gamma'| / r \rceil \leq 2|\gamma'|/r \leq 2\big(|\gamma| + \varepsilon\big)/r,$$
we obtain that $|C|$ is at most $\varphi_L(r)|\gamma|$, up to a term vanishing with $\varepsilon$.
\end{proof}

\subsection{Large scale monotonicity}\label{Monotonicity_sec}
This section is concerned with monotonicity formulae for area density. It is heavily inspired by \cite[Lemma~4.2.2]{Mul94} and \cite[Proposition~4.3.1]{Sik94}. A filling inequality implies an isoperimetric inequality for pseudoholomorphic curves (and, more generally, quasiminimzing surfaces). Considering boundary length as the derivative of area, an isoperimetric inequality can be interpreted as a differential inequality for the area\footnote{G.~Antonelli has informed the author that this is sometimes referred to as ``De Giorgi's trick''.}. Applying Grönwall's inequality~\cite[Theorem~1.2.2]{Pac98} to this differential inequality yields a monotonicity formula for area. This is the content of Proposition~\ref{monotonicity_prp}, the main step in the proof of Theorem~\ref{existence_thm}. Alternatively, applying the Bihari-LaSalle inequality~\cite[Theorem~2.3.1]{Pac98} to the differential inequality yields a lower area bound. This is the content of Proposition~\ref{lower_area_bound_prp}, which is not used in the proof of Theorem~\ref{existence_thm}, but is included in this section because of its independent interest.\\

The absolute ($L = \emptyset$) version of the following notion is due to \cites{Ban96,Ban98}.
\begin{dfn}\label{Q-quasiminimizing}
Let $(X,g)$ be a Riemannian manifold, $L \subset X$, and $Q \geq 1$. An $n$-rectifiable current $C$ such that $\partial C$ is supported on $L$ is \textit{$Q$-quasiminimizing relative to $L$} if, for every Borel set $B \supset L$ and $n$-rectifiable current $C'$ in $X$ such that $\partial(C\res B - C')$ is supported and null-homologous within $L$, $|C\res B| \leq Q|C'|$. An $n$-rectifiable current is \textit{$Q$-quasiminimizing} if it is $Q$-quasiminimizing relative to $L = \emptyset$.
\end{dfn}

\begin{eg}[{cf. \cite{Ban98}}]\label{J-curves_are_quasiminimizing_eg}
Let $(X,\omega)$ be an exact symplectic manifold equipped with a Riemannian metric and $L\subset X$ be a Lagrangian. Suppose $J$ is an almost complex structure on $X$ for which there exists $K \geq 1$ such that, as forms,
$\omega|_\Pi \leq K \textnormal{area}_\Pi$ for every oriented tangent plane $\Pi$ and 
$\omega|_\Pi \geq K^{-1} \textnormal{area}_\Pi$ when $\Pi$ is a complex tangent plane. In particular, $J$ is tamed by $\omega$. Conversely, every tamed $J$ satisfies this condition for some $K \geq 1$ on each compact set. Let $C$ be a $J$-holomorphic curve with boundary on $L$. With $B$ and $C'$ as in Definition~\ref{Q-quasiminimizing},
$$|C\res B| \leq K (C \res B)(\omega) = K C'(\omega) \leq K^2 |C'|;$$
the equality above follows from $\omega$ being exact and $\partial(C\res B - C')$ being null-homologous in $L$.
Thus, $J$-holomorphic curves with boundary on $L$ are $K^2$-quasiminimizing relative to $L$. In particular, this holds for every almost complex structure $\wt{J}$ on the universal cover $\wt{X}$ of a closed, symplectically aspherical manifold $(X,\omega)$ which is lifted from an almost complex structure $J$ on $X$ tamed by $\omega$. We recall that a symplectic form $\omega$ is \textit{aspherical} if its integral over every sphere is zero; by the Hurewicz and universal coefficient theorems, this implies that its lift $\wt{\omega}$ to $\wt{X}$ is exact.
\end{eg}

\begin{rmk}
One could also consider the version of Definition~\ref{Q-quasiminimizing} with $C\res B - C'$ required to be null-homologous in $X$ relative to $L$. This would result in a broader notion of quasiminimizing and would remove the need for the exactness assumption in Example~\ref{J-curves_are_quasiminimizing_eg}. However, this would impose a stricter condition on usable $C'$. In our setting, we know little about $C'$ (cf. Remark~\ref{homological_assumption_rel_filling_rmk}) and therefore use the stronger notion of $Q$-quasiminizing current.
\end{rmk}

\begin{lmm}[Isoperimetric~Inequality]\label{isoperimetric_lmm}
Let $(X,g)$ be a Riemannian manifold, $C$ be a 2-rectifiable current, and $Q \geq 1$. For a nonempty subset $S \subset X$, define
\begin{equation}\label{Isoperi_length_area_eqn}
\ell_S(r) := |C_r(d_S)| \quad \textnormal{and} \quad A_S(r) := |C\res B_r(S)|,
\end{equation}
where $C_r(d_S)$ are the $1$-rectifiable currents of Lemma~\ref{slicing_lmm}.
\begin{enumerate}[label=(\arabic*)]
\item\label{abs_isoperimetry_lmm} If $p \in X$, $C$ is $Q$-quasiminimizing, and $\varphi$ is a filling function for $(X,g)$, then
$$A_p(r) \leq Q \varphi(\ell_p(r))$$
for almost every $r \in \R$.
\item\label{rel_isoperimetry_lmm} If $R \in (0,\infty{]}$, $L \subset X$ is nonempty, $C$ is $Q$-quasiminimizing relative to $L$, and $\varphi_L$ is a filling function $R$-relative to $L$, then
$$A_L(r) \leq Q \varphi_L(r) \ell_L(r)$$
for almost every $r < R$.
\end{enumerate}
\end{lmm}

\begin{proof}
\ref{abs_isoperimetry_lmm} By Definition~\ref{filling_function}, for every $\varepsilon > 0$ there is a chain $C'$ with
\begin{equation}\label{iso_ineq_1}
\partial C' = C_r(d_p) \quad \textnormal{and} \quad |C'| < \varphi(\ell_p(r))+\varepsilon.
\end{equation}
By the second statement in (\ref{slicing_formulas}), $\partial\big(C\res B_r(p)\big) = C_r(d_p)$. Along with the $Q$-quasiminimality of $C$, (\ref{iso_ineq_1}) thus gives
$$A_p(r) \leq Q|C'| < Q \big(\varphi\big(\ell_p(r)\big) +\varepsilon\big).$$
As $\varepsilon$ is arbitrary, the claimed inequality follows.\\

\ref{rel_isoperimetry_lmm} Let $\eta$ be as in (\ref{retraction_eqn}). Since $\partial \big( C\res B_r(L)\big) = \partial C + C_r(d_L)$ by the second equation in~(\ref{slicing_formulas}), $\eta\big([C_r(d_L)] \big) = -[\partial C]$. By Definition~\ref{relative_filling_function}, there thus exist a cycle $\gamma'$ in $L$ and a $2$-rectifiable current $C'$ in $X$ such that
\begin{equation}\label{iso_ineq_2}
[\gamma'] = -[\partial C] \in H_1(L;\Z), \quad \partial C' = C_r(d_L) - \gamma', \quad \textnormal{and} \quad |C'| < \varphi_L(r)\ell_L(r) + \varepsilon.
\end{equation}
Therefore, $\partial(C\res B_r(L) - C')$ is supported and null-homologous within $L$. Along with the $Q$-quasiminimality of $C$ relative to $L$, (\ref{iso_ineq_2}) thus gives
$$A_L(r) \leq Q|C'| < Q(\varphi_L(r)\ell_L(r) + \varepsilon).$$
As $\varepsilon$ is arbitrary, the claimed inequality follows.
\end{proof}

\begin{prp}[{Monotonicity}]\label{monotonicity_prp}
Let $(X,g)$ be a Riemannian manifold, $L \subset X$, $R \in (0,\infty{]}$, $\varphi_L$ be an increasing filling function $R$-relative to $L$, and $C$ be a 2-rectifiable current $Q$-quasiminimizing relative to $L$. Then,
\begin{equation}\label{monotonicity_formula}
|C\res B_b(L)| \geq |C \res B_a(L)| \exp\Big(\frac{1}{Q}\int_a^b dr/\varphi_L(r)\Big)
\end{equation}
for any $0 < a < b < R$.
\end{prp}

As explained in Remark~\ref{asymptotics_rmk}, we apply Proposition~\ref{monotonicity_prp} in the setting where $R$ may be taken arbitrarily large, with $\varphi_L$ and $Q$ fixed. In this case, one can consider (\ref{monotonicity_formula}) with $a$ fixed and $b$ arbitrarily large. If $1/\varphi_L$ is not integrable at infinity (i.e. $\varphi_L$ grows slowly), then the exponential term in (\ref{monotonicity_formula}) tends to infinity with $b$, yielding arbitrarily strong inequalities; this is ultimately the reason for the non-summability criterion in Remark~\ref{nonsummable_rmk}. On the other hand, (\ref{monotonicity_formula}) is of limited utility if $1/\varphi_L$ is integrable at infinity.

\begin{proof}[\bf{\emph{Proof of Proposition~\ref{monotonicity_prp}}}]
We follow the notation in (\ref{Isoperi_length_area_eqn}). Define the functions
$$f(r) := \exp\Big(\frac{1}{Q}\int_a^r dt/\varphi_L(t)\Big) \quad \textnormal{and} \quad \wt{A_L}(r) := A_L(a) + \int_a^r \ell_L(t)dt.$$
By definition, $f' = f/Q\varphi_L$. By the inequality in (\ref{slicing_formulas}), $\wt{A_L}(r) \leq A_L(r)$ for almost every $r \geq a$. Along with Lemma~\ref{isoperimetric_lmm}\ref{rel_isoperimetry_lmm}, this gives $\wt{A_L}(r) \leq Q \varphi_L(r) \wt{A_L}'(r)$ for almost every $r \geq a$. We note that both $\wt{A_L}$ and $f$ have derivatives in $L^1_{loc}$. Therefore, almost everwhere on the interval $(a,R)$,
$$(\wt{A_L}/f)' = \frac{f\wt{A_L}' - \wt{A_L}f'}{f^2} \geq \frac{\wt{A_L}f - \wt{A_L}f}{Q\varphi_L f^2} = 0.$$
Therefore, $\wt{A_L}/f$ is monotone on $(a,R)$ and
$$A_L(b) \geq \wt{A_L}(b) \geq \wt{A_L}(a)\big(f(b)/f(a)\big) = A_L(a)f(b),$$
as claimed.
\end{proof}

\begin{prp}[{Lower Area Bound}]\label{lower_area_bound_prp}
Let $(X,g)$ be a Riemannian manifold, $\varphi$ be a strictly increasing filling function for $X$, and $C$ be a $Q$-quasiminimizing 2-rectifiable current.
Let $p\in X$ and $a >0$ be such that $A_p(a) := |C \res B_a(p)| > 0$.
Define
$$f(s) := \int_{A_p(a)}^s \frac{dt}{\varphi^{-1}(t/Q)}\, .$$
Then, for any $b > a$,
$$f\big(|C \res B_b(p)|\big) \geq b-a.$$
\end{prp}

\begin{proof}
Since $\varphi$ is strictly increasing, it is invertible and thus $f$ is well-defined. Similarly, as $\varphi$ is positive, $f$ is increasing. With notation as in (\ref{Isoperi_length_area_eqn}), let
$$\wt{A_p}(r) := A_p(a) + \int_a^r \ell_p(t) dt.$$
By the inequality in (\ref{slicing_formulas}), $\wt{A_p}(r) \leq A_p(r)$ for $r \geq a$. Along with Lemma~\ref{isoperimetric_lmm}\ref{abs_isoperimetry_lmm}, this gives $\wt{A_p}(r) \leq Q \varphi(\wt{A_p}'(r))$ for almost every $r \geq a$. As $f$ is Lipschitz on $(A_p(a),\infty)$ and $\wt{A_p}$ has a derivative in $L^1_{loc}$ on the interval $(a,\infty)$,
\begin{equation}\label{chain_rule_calc}
\big(f(\wt{A_p}(r))\big)' = f'\big(\wt{A_p}(r)\big)\wt{A_p}'(r) = \wt{A_p}'(r)/\varphi^{-1}\big(\wt{A_p}(r)/Q\big) \geq \wt{A_p}'(r)/\wt{A_p}'(r) = 1
\end{equation}
almost everywhere on $(a,\infty)$.
Since $f\big(\wt{A_p}(a)\big) = 0$ and $f$ is increasing, integrating (\ref{chain_rule_calc}) from $a$ to $b$ gives
$$f\big(A_p(b)\big) \geq f\big(\wt{A_p}(b)\big) \geq b-a,$$
as claimed.
\end{proof}

\begin{eg}
Let $C$ be a $Q$-quasiminimizing surface passing through a point $p$ in a Euclidean space. In this case, $|C\res B_r(p)| > 0$ for all $r > 0$ and (by taking a cone) we may take $\varphi(r) = r^2/2$. Proposition~\ref{lower_area_bound_prp} then implies that $|C\res B_r(p)| \geq r^2/2Q$. Let $(X,J,g)$ be a closed almost Hermitian manifold. Every point has a neighborhood on which $J$ is tamed by an exact symplectic form and $g$ is approximately the Euclidean metric. Along with Example~\ref{J-curves_are_quasiminimizing_eg}, this implies that, for $r > 0$ sufficiently small, the amount of area of a closed $J$-holomorphic curve contained in a ball of radius $r$ around one of its points is bounded below by a constant times $r^2$ (cf. \cite[Proposition~4.3.1]{Sik94}).
\end{eg}

\begin{lmm}\label{isoperimetry_from_coarse_isoperimetry_lmm}
Let $(D,g)$ be a disk with a Riemannian metric. There exists an embedded disk $D'$ with rectifiable boundary such that $\partial D' \subset B_1(\partial D)$ and
\begin{equation}\label{iso_coarse_iso_ineq}
\frac{\textnormal{length}(\partial D')}{\textnormal{area}(D')} \leq 
\frac{\textnormal{area}(B_1(\partial D))}{\textnormal{area}(D - B_1(\partial D))}\, .
\end{equation}
\end{lmm}

\begin{proof}
Let $f: D \lra \left[ 0,\infty \right)$ be the distance function to the boundary. It is clearly 1-Lipschitz. As $\partial D$ is connected, it follows that the sublevel sets $f^{-1}\big([0,r]\big)$ are connected. Therefore, the connected components of the superlevel sets $f^{-1}\big((r,\infty)\big)$ are simply connected \cite[Chapter~4.4.2, Definition~1]{Ahl78} and thus are disks \cite[Chapter 6.1.1, Theorem~1]{Ahl78}. By the refined coarea formula in \cite[Theorem~2.5.iv]{Alb13}, there exists $r \in (0,1)$ such that $f^{-1}(r)$ is a disjoint union of rectifiable Jordan curves and points, with total length at most $\textnormal{area}(B_1(\partial D))$. A subset of this set bounds the superlevel set $S := f^{-1}\big((r,\infty)\big)$, which has area at least $\textnormal{area}(D - B_1(\partial D))$. Therefore,
$$\frac{\textnormal{length}(\partial S)}{\textnormal{area}(S)} \leq 
\frac{\textnormal{area}(B_1(\partial D))}{\textnormal{area}(D - B_1(\partial D))} \, .$$
Since $S$ is a disjoint union of disks with disjoint boundaries, at least one of these disks $D'$ also satisfies the claimed inequality.
\end{proof}

\begin{rmk}\label{discrete_zeros_rmk}
By approximating, one sees that Lemma~\ref{isoperimetry_from_coarse_isoperimetry_lmm} applies even if the Riemannian metric has isolated zeros. For instance, it applies to the metric induced on the disk by a map to a Riemannian manifold which is an immersion outside of a discrete set. By~\cite[Lemma~2.4.1]{McD12}, Lemma~\ref{isoperimetry_from_coarse_isoperimetry_lmm} thus applies to the metric induced by a nonconstant pseudoholomorphic map.
\end{rmk}

\begin{rmk}\label{asymptotic_isoperimetry_rmk}
As explained after Proposition~\ref{monotonicity_prp}, if the growth of the filling function $\varphi_L$ in Proposition~\ref{monotonicity_prp} is at most linear (or, more generally, $1/\varphi_L$ is not integrable at infinity), then the exponential in (\ref{monotonicity_formula}) increases to infinity with $b$. Combined with (\ref{iso_coarse_iso_ineq}), this implies that, if $\varphi_L$ is a filling function $R$-relative to a subset $L$, then every $Q$-quasiminimizing disk $D$ with boundary on $L$ contains a slightly smaller disk $D' \subset D$ satisfying
\begin{equation}\label{reverse_iso_ineq}
\textnormal{area}(D') \geq C(R)\textnormal{length}(\partial D')
\end{equation}
for a function $C: (0,\infty) \lra (0,\infty)$ increasing to infinity which depends only on $\varphi_L$ and $Q$. 
In the setting of Example~\ref{J-curves_are_quasiminimizing_eg}, this yields a version of the reverse isoperimetric inequality of \cite{Gro14, Duv16} that is suitable for establishing Theorem~\ref{existence_thm}.
\end{rmk}

\begin{rmk}
Proposition~\ref{monotonicity_prp} does not imply (\ref{reverse_iso_ineq}) directly, because we cannot control the exponential in (\ref{monotonicity_formula}) in the limit as $a$ tends to zero. Therefore, we rely on Lemma~\ref{isoperimetry_from_coarse_isoperimetry_lmm} to derive (\ref{reverse_iso_ineq}) and ultimately construct an Ahlfors current. As (\ref{iso_coarse_iso_ineq}) only applies on a smaller disk, so does (\ref{reverse_iso_ineq}) and we therefore cannot control the homology class of the resulting Ahlfors current. However, with more geometric control near $L$, one could improve the estimate in Proposition~\ref{monotonicity_prp} in the limit as $a$ tends to $0$ and thus avoid the appeal to Lemma~\ref{isoperimetry_from_coarse_isoperimetry_lmm}. This would retain control over the homology class of the Ahlfors current constructed in the proof of Theorem~\ref{existence_thm}.
\end{rmk}

\subsection{Continuity method}\label{existence_of_disks_sec}
This section contains the proofs of Theorem~\ref{existence_thm} and Corollaries~\ref{surfaces_crl}, \ref{nonpositive_crl}, and \ref{group_crl}. 
The following lemma, which is proved in the same manner as \cite[Theorem 2.3.C]{Gro85}, provides disks to which one can apply Proposition~\ref{monotonicity_prp}.
\begin{lmm}\label{existence_of_disks_lmm}
Let $(X,\omega_X)$ be a closed symplectic manifold and $L_X$ be a Lagrangian in $X$ such that $\omega_X$ vanishes on $\pi_2(X,L_X)$. Let $(T,\omega_T)$ be a symplectic 2-torus and $S$ be a circle bounding a disk $D$ in $T$. Then, for every almost complex structure $J$ on $X \times T$ tamed by $\omega := \omega_X \oplus \omega_T$, the Lagrangian $L := L_X \times S$ bounds a $J$-holomorphic disk. 
\end{lmm}

\begin{proof}
For an almost complex structure $J$ on $X$, let $\mathfrak{M}_{0,1}(J,[D])$ denote the space of $J$-holomorphic disks homologous to $D$ relative to $L$ with one marked point on the boundary. This moduli space is compact, as is its analogue for a compact one-parameter family of almost complex structures $J$ tamed by $\omega$. Sphere bubbling cannot occur because $\omega$ is aspherical. Neither disk bubbling nor collapsing to a multiple cover can occur because the pairing of $[D]$ with $\omega$ is minimal among relative homology classes pairing positively with $\omega$. Let
$$\textnormal{ev}_J: \mathfrak{M}_{0,1}(J,[D]) \lra L$$
be the evaluation map taking such a disk to its marked point.\\

We first consider a special almost complex structure. Equip $X$ and $T$ with almost complex structures $J_X$ and $J_T$ which are tamed by $\omega_X$ and $\omega_T$. Let $J_0$ be the product almost complex structure on $X \times T$. Since $\omega_X$ vanishes on $\pi_2(X,L_X)$, the projection of a $J_0$-holomorphic disk to $X$ is constant. Therefore, each element in $\mathfrak{M}_{0,1}(J_0,[D])$ has the form $(x \!\times\! D,p)$, where $x\in X$ and $p \in \partial D$. Conversely, each such disk lies in $\mathfrak{M}_{0,1}(J_0,[D])$. Furthermore, the linearization of the Cauchy-Riemann operator on each disk is complex linear (see Lemma~\ref{split_str_lmm}). Therefore, as in \cite[Lemma~3.3.1]{McD12}, it follows that each disk is Fredholm regular. Thus, $\textnormal{ev}_{J_0}$ has degree $1\!\!\mod 2$.\\

For a generic almost complex structure $J_1$ tamed by $\omega$, $\mathfrak{M}_{0,1}(J_1,[D])$ is a smooth manifold \cite[Theorem~3.1.6]{McD12}. A generic path of almost complex structures tamed by $\omega$ connecting $J_0$ to $J_1$ provides a bordism between $\textnormal{ev}_{J_0}$ and $\textnormal{ev}_{J_1}$ \cite[Theorem~3.1.8]{McD12}. As the former has degree $1\!\!\mod 2$, the claim follows.
\end{proof}

\begin{proof}[{\bf{\emph{Proof of Theorem~\ref{existence_thm}}}}]
As $X$ retracts onto $L$, the map $\pi_1(L) \lra \pi_1(X)$ induced by inclusion is injective. It follows from the long exact sequence of relative homotopy groups that the natural map $\pi_2(X) \lra \pi_2(X,L)$ is a surjection. Therefore, as $(X,\omega)$ is symplectically aspherical, $\omega$ vanishes on $\pi_2(X,L)$. Let $J$ be an almost complex structure on $X \times T^2$ which is tamed by $\omega \oplus \omega_{std}$. Lemma~\ref{existence_of_disks_lmm} implies that $L \times S$ bounds a $J$-holomorphic disk for every embedded circle $S \subset T^2$.\\

Let $\wt{X}$ be the universal cover of $X$ and $\wt{L}$ and $\wt{J}$ be the lifts of $L$ and $J$ to $\wt{X} \times \R^2$, respectively. Let $S_R \subset \R^2$ be a circle of radius $R>0$ and
$$L_R := \wt{L} \times S_{2R} \subset \wt{X} \times \R^2.$$
There exists a finite cover $T'$ of $T^2$ such that the image of $S_{2R}$ in  $T'$ is an embedded circle $S$. By the above, $L \! \times \! S$ bounds a $J'$-holomorphic disk, where $J'$ is the lift of the almost complex structure $J$ on $X \times T^2$ to $X \times T'$. This disk lifts to a $\wt{J}$-holomorphic disk $D_R$ in $\wt{X} \times \R^2$ with boundary on $L_R$.\\

By assumption, the optimal filling function for $\wt{X}$ has at most quadratic growth. Therefore, by Lemma~\ref{product_with_space_lmm}, so does that for $\wt{X} \times \R^2$. Denote it by $\varphi$. We now show that there exists $K\geq 1$, independent of $R$, such that
$$\varphi_L := \varphi_{L_R}: {[}0,\infty{)} \lra {[}0,\infty{)}, \quad \varphi_L(r) = \frac{2\varphi(\relbound Kr)}{r}$$
is a filling function for $\wt{X} \times \R^2$ $R$-relative to $L_R$.
We may take the retraction $X \lra L$ to be smooth \cite[Theorem~6.26]{Lee13}, and therefore $K_1$-Lipschitz, for some $K_1 \geq 1$. It lifts to a Lipschitz retraction of the universal cover $\wt{X} \lra \wt{L}$.
There exists $K_2 \geq 1$ such the the radial projection $B_1(S_2) \lra S_2$ is $K_2$-Lipschitz. By scaling, the radial projection $B_{R}(S_{2R}) \lra S_{2R}$ is $K_2$-Lipschitz. Therefore, there exists $K \geq 1$, independent of $R$, such that $B_R(L_R)$ $K$-Lipschitz retracts onto $L_R$. As $\varphi$ is superadditive by Lemma~\ref{superadditivity_lmm}, it follows from Lemma~\ref{relative_from_filling_lmm} that $\varphi_L$ is a filling function for $\wt{X} \times \R^2$ $R$-relative to $L_R$.\\

By Example~\ref{J-curves_are_quasiminimizing_eg}, there exists $Q \geq 1$, independent of $R$, such that $D_R$ is $Q$-quasiminimizing relative to $L_R$.
As $\varphi$ has at most quadratic growth, $\varphi_L$ has at most linear growth. Therefore, by Proposition~\ref{monotonicity_prp},
\begin{equation}\label{asymptotic_coarse_ineq}
\lim_{R \rightarrow \infty}\frac{|D_R \res B_1(L_R)|}{|D_R|} = 0.
\end{equation}
Let $u: D_R \lra \wt{X} \times \R^2$ be a $\wt{J}$-holomorphic parameterization of $D_R$. With respect to the pullback metric, $B_1(\partial D_R) \subset u^{-1}\big(B_1(L_R)\big)$ and $\textnormal{area}(D_R) = |D_R|$. Therefore, by (\ref{asymptotic_coarse_ineq}),
$$\lim_{R \rightarrow \infty} \frac{\textnormal{area}(B_1(\partial D_R))}{\textnormal{area}(D_R)} = 0.$$
Along with Lemma~\ref{isoperimetry_from_coarse_isoperimetry_lmm}, this implies that there exists a family of $\wt{J}$-holomorphic disks $D'_R \subset D_R$ such that
$$\lim_{R \rightarrow \infty} \frac{\textnormal{length}(\partial D'_R)}{\textnormal{area}(D'_R)} = 0,$$
which yields an Ahlfors current on $X \times T^2$ by \cite[Lemma~26.14]{Sim83}. The claim now follows from \cite[Théorème]{Duv08}.
\end{proof}

\begin{proof}[{\bf{\emph{Proof of Corollary~\ref{surfaces_crl}}}}]
If at least one of the genera is $0$, the claim follows from \cite[Theorem~2.3.C]{Gro85}. Therefore, assume no genus is $0$. If at least one of the genera is $1$, then consider the product of the surfaces, excluding that one. Denote it by $X'$. It admits a metric of nonpositive curvature, so by Example~\ref{nonpositive_filling_function_eg}, $\delta_{X'}$ grows at most quadratically. Furthermore, each surface of genus at least $1$ retracts onto a closed curve. Take the product of these closed curves and denote it by $L$. The submanifold $L$ is a closed Lagrangian onto which $X'$ retracts and the claim follows from Theorem~\ref{existence_thm}.\\

Suppose that all of the genera are at least $2$. Then, each of their symplectic forms is hyperbolic in the sense of Definition~\ref{hyperbolic_form_dfn} \cite[Example~0.2.C']{Gro91}. Therefore, there exist $C > 0$ and $\beta$ such that $d\beta = \pi^* \omega$ on $\wt{X}$ and $|\beta| < C$. Let $J$ be an almost complex structure on $X$ tamed by $\omega$. Then, there exists $C' > 0$ such that, for any $J$-holomorphic disk $D$ in $X$, $\textnormal{area}(D) < C' \int_D \omega$. A disk in $X$ lifts to the universal cover $\wt{X}$ and by Stokes' theorem satisfies
$$\int_D \omega = \int_{\partial D} \beta \leq C \cdot \textnormal{length}(\partial D).$$
Therefore, $J$-holomorphic disks in $X$ satisfy a linear isoperimetric inequality. Thus, $(X,J)$ contains no Ahlfors currents. By the Ahlfors lemma \cite[Lemma~6.9]{Gro99} (see also \cites{Duv17}), $(X,J)$ is therefore hyperbolic.
\end{proof}

\begin{proof}[{\bf{\emph{Proof of Corollary~\ref{nonpositive_crl}}}}]
By Example~\ref{nonpositive_filling_function_eg}, $\delta_X$ grows at most quadratically. The lift $\wt{L}$ of $L$ to the universal cover $\wt{X}$ is totally geodesic with respect to the lifted metric. The normal exponential map provides a $1$-Lipschitz retraction by \cite[Chapter~X, Theorems~2.4 and 2.5]{Lan99}. The claim follows from Theorem~\ref{existence_thm} and Remark~\ref{nonsummable_rmk}.
\end{proof}

\begin{proof}[{\bf{\emph{Proof of Corollary~\ref{group_crl}}}}]
The diagonal in $(\textnormal{Sym}^n(\Sigma_g) \times \textnormal{Sym}^n(\Sigma_g), \omega \oplus -\omega)$ is a closed Lagrangian and the map $(x,y) \mapsto (x,x)$ is a retraction onto the diagonal. By \cite[Corollary 3.5 and Theorem~4.1]{DiC25}, $(\textnormal{Sym}^n(\Sigma_g), \omega)$ has free abelian fundamental group and is symplectically aspherical if $g \geq 2n-1$. Therefore, the claim follows from Corollary~\ref{free_abelian_crl} and Theorem~\ref{existence_thm}.
\end{proof}

\section{Bubbling Ahlfors currents}\label{bubbling_sec}
This section is devoted to the proof of Theorem~\ref{bubbling_thm} and Corollaries~\ref{GW_crl} and \ref{partial_hyperbolicity_crl}. The proof of Theorem~\ref{bubbling_thm} is a short averaging argument. The key point is the equidistribution phenomenon expressed in Lemma~\ref{equidistribution_circles_lmm}.\\

Let $\Sigma$ be a closed hyperbolic surface and $\pi: \wt{\Sigma} \lra \Sigma$ be its universal cover by the hyperbolic plane. For $p \in \wt{\Sigma}$ and $r \in (0,\infty)$, let $B_r(p)$ be the disk of radius $r$ centered at $p$ and $\gamma_r(p) := \partial B_r(p)$ be the circle of radius $r$ centered at $p$, both taken with respect to the hyperbolic metric.

\begin{lmm}[{\cite[Theorem~2.1]{Esk93}}]\label{equidistribution_circles_lmm}
Let $f$ be a continuous function on $\Sigma$. The average of $f \circ \pi$ over $\gamma_r(p)$ (taken with respect to the hyperbolic length) tends to the average of $f$ over $\Sigma$ (taken with respect to the hyperbolic area) as $r$ tends to infinity.
\end{lmm}
Lemma~\ref{equidistribution_circles_lmm} immediately implies the following via integrating in polar coordinates (equivalently, radially decomposing the area measure).

\begin{crl}\label{equidistribution_disks_crl}
The average of $f \circ \pi$ over $B_r(p)$ (taken with respect to the hyperbolic area) tends to the average of $f$ over $\Sigma$ (taken with respect to the hyperbolic area) as $r$ tends to infinity.
\end{crl}

\begin{prp}\label{averaging_prp}
Let $\Sigma$ be a closed Riemann surface of genus $G > 1$, equipped with a Hermitian metric $g$. Then, there is a sequence of holomorphic maps $u_i: D \lra \Sigma$ such that
\begin{equation}\label{averaging_inequality}
\lim_{i \rightarrow \infty} \frac{\textnormal{length}\big(u_i(\partial D)\big)}{\textnormal{area}\big(u_i(D)\big)} \leq \sqrt{\frac{4\pi(G-1)}{\textnormal{area}(\Sigma)}}
\end{equation}
and $u_i(D)/\textnormal{area}(u_i(D))$ tends to $\Sigma/\textnormal{area}(\Sigma)$ weakly as a current.
\end{prp}

\begin{proof}
Let $h$ be the hyperbolic metric of constant curvature $-1$ on $\Sigma$ given by the uniformization theorem \cite[Theorem~XII.0.1]{deS16} and $\lambda: \Sigma \lra (0,\infty)$ be the smooth function such that $g = \lambda^2 h$. Let $\pi: \wt{\Sigma} \lra \Sigma$ be the universal cover by the hyperbolic plane, as above, and $p \in \wt{\Sigma}$. We will show that the sequence
$$u_i := \pi: B_i(p) \lra \Sigma$$
satisfies the claimed properties.\\

Let $\ell_g$ and $\ell_h$ be lengths taken with respect to $g$ and $h$ (or their pullbacks to $\wt{\Sigma}$), respectively. Let $A_g$ and $A_h$ be the areas taken with respect to $g$ and $h$ (or their pullbacks to $\wt{\Sigma}$), respectively. For a function $f: X \lra \R$, let $\avg(f,\Sigma)$, $\avg(f,B_r)$, and $\avg(f,\gamma_r)$ be averages of $f$ (with respect to the relevant hyperbolic measure) over $\Sigma$, $B_r$, and $\gamma_r$, respectively. By \cite[Section~13, Fact~5]{Can97},
$$\lim_{i \rightarrow \infty} \frac{\ell_h\big(\gamma_i(p)\big)}{A_h\big(B_i(p)\big)} = 1.$$
Therefore, by Lemma~\ref{equidistribution_circles_lmm}, Corollary~\ref{equidistribution_disks_crl}, and Jensen's (or H{ö}lder's) inequality,
\begin{equation*}
\begin{split}
\lim_{i \rightarrow \infty} \frac{\ell_g\big(u_i(\partial D)\big)}{A_g\big(u_i(D)\big)}
= &\lim_{i \rightarrow \infty} \frac{\ell_g\big(\gamma_i(p)\big)}{A_g\big(B_i(p)\big)}
= \lim_{i \rightarrow \infty} \frac{\avg\big(\lambda,\gamma_i(p)\big) \ell_h\big(\gamma_i(p)\big)}{\avg\big(\lambda^2,B_i(p)\big) A_h\big(B_i(p)\big)}\\
= &\frac{\avg(\lambda, \Sigma)}{\avg(\lambda^2,\Sigma)}
\leq  \frac{1}{\sqrt{\avg(\lambda^2,\Sigma)}}
= \sqrt{\frac{A_h(\Sigma)}{A_g(\Sigma)}}
= \sqrt{\frac{4\pi(G-1)}{A_g(\Sigma)}}\, ;\\
\end{split}
\end{equation*}
the last equality follows from the Gauss-Bonnet theorem~\cite[Chapter~4-5, Global Gauss-Bonnet Theorem]{doC76}. Therefore, (\ref{averaging_inequality}) holds. The last claim in Proposition~\ref{averaging_prp} follows immediately from Corollary~\ref{equidistribution_disks_crl}.
\end{proof}
\begin{rmk}\label{bubbling_discrete_zeroes_rmk}
Similarly to Lemma~\ref{isoperimetry_from_coarse_isoperimetry_lmm} (cf. Remark~\ref{discrete_zeros_rmk}), Proposition~\ref{averaging_prp} applies also to the case in which the Riemannian metric $g$ has discrete zeroes and therefore to the metric induced by a nonconstant pseudoholomorphic map.
\end{rmk}

\begin{proof}[{\bf{\emph{Proof of Theorem~\ref{bubbling_thm}}}}]
By \cite[Lemma~26.14]{Sim83}, a subsequence of $\Sigma_i/\textnormal{area}(\Sigma_i)$ converges to a current $T$. By Proposition~\ref{averaging_prp}, each $\Sigma_i/\textnormal{area}(\Sigma_i)$ is a limit of $D_{ij}/\textnormal{area}(D_{ij})$ with $$\lim_{i \rightarrow \infty} \lim_{j \rightarrow \infty} \frac{\textnormal{length}(\partial D_{ij})}{\textnormal{area}(D_{ij})}=0.$$
As a limit of limits is a limit in the weak topology on currents of finite mass \cite[Corollary~7.3]{Fed60}, $T$ is an Ahlfors current.
\end{proof}

\begin{proof}[{\bf{\emph{Proof of Corollary~\ref{GW_crl}}}}]
Suppose $\Sigma_i$ is a sequence of closed pseudoholomorphic curves with fixed genus and homology class, but unbounded area. By Theorem~\ref{bubbling_thm}, a subsequence of $\Sigma_i/\textnormal{area}(\Sigma_i)$ converges to an null-homologous Ahlfors current. Therefore, if no Ahlfors current is null-homologous, there is a uniform upper bound on the area of a pseudoholomorphic curve depending only on its genus and homology class. By Gromov's compactness theorem~\cite{Gro85}, the moduli spaces $\ov{\mathfrak{M}}_g(A;J)$ are compact. The Gromov-Witten type invariants can then be defined using virtual techniques as in~\cite{Li98,Par16}.
\end{proof}

The next lemma follows immediately from Stokes' Theorem and the fact that disks lift to universal covers.

\begin{lmm}\label{hyperbolic_form_lmm}
Let $(X,J)$ be a closed almost complex manifold and $T$ be an Ahlfors current on $X$. Then, $T(\alpha) = 0$ for every hyperbolic 2-form $\alpha$.
\end{lmm}

\begin{proof}[{\bf{\emph{Proof of Corollary~\ref{partial_hyperbolicity_crl}}}}]
Suppose there exists $\varepsilon > 0$ such that for each $i \in \Z^+$, there exists $\Sigma_i$ such that $\int_{\Sigma_i} \alpha > \varepsilon \int_{\Sigma_i} \omega$ and $\int_{\Sigma_i} \omega \geq i\cdot \textnormal{genus}(\Sigma_i)$. Theorem~\ref{bubbling_thm} then yields an Ahlfors current $T$ such that $T(\alpha) \neq 0$, contradicting Lemma~\ref{hyperbolic_form_lmm}.
\end{proof}

\section{Flexibility of disks}\label{convexity_sec}
This section contains the proofs of Theorem~\ref{convexity_thm} and Example~\ref{extreme_points_eg}. The proof of the former consists of a gluing argument (Proposition~\ref{reduction_to_surface_case_prp}) followed by a direct construction. The key point is that we are free to add a local perturbation to achieve regularity. A similar idea was exploited to great effect in \cite{Suk12}. We take a different approach which works for any bordered Riemann surface and uses the setup standard in symplectic geometry.\\

Let $(X,J)$ be an almost complex manifold. 
We call a half-dimensional submanifold $Y\!\subset\!X$ \textit{totally real} 
if \hbox{$TY\!\cap\!J(TY)\!=\!Y$}. 
For a compact, connected Riemann surface $\Sigma$, possibly nodal and with boundary, we denote by $\Sigma^* \subset \Sigma - \partial \Sigma$ the subspace of smooth points. For a smooth map $u: (\Sigma,\partial \Sigma) \lra (X,Y)$, define
\begin{equation*}
\begin{split}
\Gamma(u;Y)\!&:=\!\big\{\xi\!\in\!\Gamma(\Sigma;u^*TX):\!
\xi|_{\partial\Sigma}\!\in\!\Gamma\big(\partial\Sigma;\{\partial u\}^*TY\big)\big\}
\quad \textnormal{and}\\
\Gamma^{0,1}_J(u)\!&:=\!\Gamma\big(\Sigma;(T^*\Sigma)^{0,1}\!\otimes_{\C}\!u^*TX\big).
\end{split}
\end{equation*}
If $u$ is $J$-holomorphic, we denote by
\begin{equation*}
D_{Y;J;u}\!:\Gamma(u;Y) \lra \Gamma^{0,1}_J(u)
\end{equation*}
the linearization of the $\overline{\partial}_J$-operator on the space of smooth maps 
from~$(\Sigma,\partial\Sigma)$ to~$(X,Y)$ at $u$, as in \cite[Section~2]{Iv98}. 

\begin{lmm}[{\cite[Equation~2.2.1]{Iv98}}]\label{split_str_lmm}
The operator $D_{Y;J;u}$ is a real Cauchy-Riemann operator. The $J$-antilinear part of $D_{Y;J;u}$ has the form
$$\bigl\{R(\xi)\bigr\}(v) = N(\xi,\textnormal{d}u(v)) \quad \forall \xi \in \Gamma(u;Y), v \in T\Sigma,$$
where $N$ is the Nijenhuis tensor of $J$ (up to a factor depending on the normalization of the latter).
\end{lmm}

The following is proved in \cite{Zin17}, for example. The ability to choose the support follows from the unique continuation property of elements of the kernel of the adjoint operator.

\begin{lmm}[{\cite[Lemma~4.1]{Zin17}}]\label{perturbation_lmm}
Let $U \subset \Sigma^*$ be an open set intersecting every irreducible component of $\Sigma$. Then, there exists a finite-dimensional subspace $S_U \subset \Gamma^{0,1}_J(u)$ such that each section in $S_U$ is supported in $U$ and $\textnormal{Image}(D_{Y;J;u}) + S_U = \Gamma^{0,1}_J(u)$.
\end{lmm}

\begin{prp}\label{reduction_to_surface_case_prp}
Let $(X,J)$ be a connected almost complex manifold and $u_1, u_2: D \lra X$ be $J$-holomorphic maps from the unit disk. Let $U_1, U_2 \subset D$ be open subsets. Then, there are sequences of connected Riemann surfaces $\Sigma_i$, embeddings
$$\iota_{i}: (D-U_1) \sqcup (D-U_2) \lra \Sigma_i,$$
and $J$-holomorphic maps $w_i: \Sigma_i \lra X$ such that the maps $w_i \circ \iota_i$ converge in the $C^1$ sense to $u_1$ on $D-U_1$ and $u_2$ on $D-U_2$.
\end{prp}

\begin{proof}
We first assume that $u_1$ and $u_2$ are embeddings. By \cite[Theorem~1.1]{Suk12}, there is an embedded $J$-holomorphic disk passing through any sufficiently close pair of points. As $X$ is connected, any pair of points is connected by a finite sequence of $J$-holomorphic disks. Take such a sequence for a point $p \in u_1(U_1)$ and a point $q \in u_2(U_2)$. This yields a connected, nodal Riemann surface $\Sigma$ including $D_1 = \textnormal{Dom}(u_1)$ and $D_2 = \textnormal{Dom}(u_2)$, with no nodes in $D_1 - U_1 \subset D_1$ or $D_2 - U_2 \subset D_2$ and a $J$-holomorphic map $u: \Sigma \lra X$ such that $u|_{D_1} = u_1$, $u|_{D_2} = u_2$, and $u$ is an embedding on $\partial \Sigma$. Take a totally real submanifold $Y$ such that $u(\partial \Sigma) \subset Y$. Let $U \subset \Sigma^*$ be an open subset intersecting every irreducible component of $\Sigma$ and such that $U \cap D_1 \subset U_1$, $U \cap D_2 \subset U_2$, and $\Sigma - U$ is connected. Let $S_U \subset \Gamma(u;Y)$ be a family of perturbations as in Lemma~\ref{perturbation_lmm}. Extend $S_u$ $C^1$-smoothly to smooth maps $\Sigma' \lra X$ that are $C^1$-close to $u$ on compact subsets of $\Sigma^*$.\\

By the gluing construction in the proof of \cite[Proposition~3.4]{Li98}, there exist sequences of irreducible Riemann surfaces $\Sigma'_i$ degenerating to $\Sigma$ and of maps $w_i: \Sigma'_i \lra X$ satisfying $\overline{\partial}w_i \in S_U$ which $C^1$ converge to $u$ away from the nodes. The surfaces $\Sigma_i := \Sigma'_i - U$ are connected, the maps $w_i|_{\Sigma_i}$ are $J$-holomorphic, and the maps $w_i|_{D_1-U_1}$ and $w_i|_{D_2-U_2}$ converge to $u_1|_{D_1-U_1}$ and $u_2|_{D_2-U_2}$, respectively, as claimed.\\

For $u_1$ and $u_2$ not necessarily embeddings, we use a graph trick. Consider $X \times \C$ with the product almost complex structure, $J_{prod}$. Let $D_1$ and $D_2$ be disjoint disks in $\C$ and let $D'_j$ be the graph of $u_j$ over $D_j$ for $j = 1,2$. The disks $D'_j$ are then embedded and $J_{prod}$-holomorphic, so the argument above applies.
\end{proof}

\begin{lmm}\label{spectacles_lmm}
Let $\Sigma$ be a connected Riemann surface. Let $D_1, D_2 \subset \Sigma$ be disjoint embedded disks and $t \in [0,1]$. Then, there exists a sequence of holomorphic maps $u_i: D \lra \Sigma$ such that
\begin{equation}\label{spectacles_formula}
\lim_{i \rightarrow \infty} \frac{\textnormal{length}\big(u_i(\partial D)\big)}{\textnormal{area}\big(u_i(D)\big)} \leq
\max_{j \in \{1,2\}} \frac{\textnormal{length}(\partial D_j)}{\textnormal{area}(D_j)}
\end{equation}
and the currents $u_i(D)/\textnormal{area}(u_i(D))$ tend weakly to
$$t(D_1/\textnormal{area}(D_1)) + (1-t)(D_2/\textnormal{area}(D_2)).$$
\end{lmm}

\begin{proof}
Removing an arbitrarily small disk from the center of a disk $D$ yields an annulus with area and boundary length arbitrarily close to those of $D$. Therefore, we prove the claim holds for annuli $A_1, A_2 \subset \Sigma$, and the statement for disks follows. Let $m_i$ and $n_i$ be two sequences of positive integers, both tending to infinity, such that
\begin{equation}\label{rational_approx_eqn}
\lim_{i \rightarrow \infty} \frac{m_i}{n_i} = \frac{t\cdot \textnormal{area}(A_2)}{(1-t)\cdot \textnormal{area}(A_1)}\, .
\end{equation}

As $\Sigma$ is connected, there is a holomorphic strip $h \subset \Sigma$ such that $S := A_1 \cup h \cup A_2$ is connected. For each $i$, let $u_i: D \lra S$ be a holomorphic map that wraps $m_i$ times around $A_1$, crosses over~$h$, then wraps $n_i$ times around $A_2$; see Figure~\ref{spectacles_figure}. We will show that the sequence $u_i$ has the claimed properties.\\

We first describe the map $u_i$ more precisely. Let $\pi: \wt{A_1} \sqcup \wt{A_2} \lra A_1 \sqcup A_2$ be the universal covering. The strip $h$ can be glued to $\wt{A_1} \sqcup \wt{A_2}$ to yield a connected Riemann surface $H$. The map $\pi$ extends across $h$ as the identity to yield a map $\pi: H \lra X$. Take a connected subset $D_i \subset H$ consisting of $m_i$ fundamental domains of $\wt{A_1}$, the strip $h$, and $n_i$ fundamental domains of $\wt{A_2}$. Assume the fundamental domains and $h$ have all been chosen at the start to have finite perimeter and area. Then, we define $u_i := \pi|_{D_i}$.\\

Up to an additive constant,
\begin{equation*}
\begin{split}
\textnormal{area}(u_i) &=  m_i \textnormal{area}(A_1) + n_i \textnormal{area}(A_2) \quad \textnormal{and}\\ \textnormal{length}(\partial u_i) &=  m_i \textnormal{length}(\partial A_1) + n_i \textnormal{length}(\partial A_2).
\end{split}
\end{equation*}
It follows that (\ref{spectacles_formula}) holds. Furthermore, up to an error which vanishes in the limit,
$$\frac{u_i(D)}{\textnormal{area}(u_i(D))} = \frac{m_i A_1 + n_i A_2}{m_i \textnormal{area}(A_1) + n_i \textnormal{area}(A_2)}\, .$$
By (\ref{rational_approx_eqn}), this tends to $t(A_1/\textnormal{area}(A_1)) + (1-t)(A_2/\textnormal{area}(A_2))$.
\end{proof}

\begin{figure}
\centering

% Pattern Info
 
\tikzset{
pattern size/.store in=\mcSize, 
pattern size = 5pt,
pattern thickness/.store in=\mcThickness, 
pattern thickness = 0.3pt,
pattern radius/.store in=\mcRadius, 
pattern radius = 1pt}
\makeatletter
\pgfutil@ifundefined{pgf@pattern@name@_0pa1t4syp}{
\pgfdeclarepatternformonly[\mcThickness,\mcSize]{_0pa1t4syp}
{\pgfqpoint{0pt}{0pt}}
{\pgfpoint{\mcSize+\mcThickness}{\mcSize+\mcThickness}}
{\pgfpoint{\mcSize}{\mcSize}}
{
\pgfsetcolor{\tikz@pattern@color}
\pgfsetlinewidth{\mcThickness}
\pgfpathmoveto{\pgfqpoint{0pt}{0pt}}
\pgfpathlineto{\pgfpoint{\mcSize+\mcThickness}{\mcSize+\mcThickness}}
\pgfusepath{stroke}
}}
\makeatother

% Pattern Info
 
\tikzset{
pattern size/.store in=\mcSize, 
pattern size = 5pt,
pattern thickness/.store in=\mcThickness, 
pattern thickness = 0.3pt,
pattern radius/.store in=\mcRadius, 
pattern radius = 1pt}
\makeatletter
\pgfutil@ifundefined{pgf@pattern@name@_4dhc1fzc0}{
\pgfdeclarepatternformonly[\mcThickness,\mcSize]{_4dhc1fzc0}
{\pgfqpoint{0pt}{0pt}}
{\pgfpoint{\mcSize+\mcThickness}{\mcSize+\mcThickness}}
{\pgfpoint{\mcSize}{\mcSize}}
{
\pgfsetcolor{\tikz@pattern@color}
\pgfsetlinewidth{\mcThickness}
\pgfpathmoveto{\pgfqpoint{0pt}{0pt}}
\pgfpathlineto{\pgfpoint{\mcSize+\mcThickness}{\mcSize+\mcThickness}}
\pgfusepath{stroke}
}}
\makeatother

% Pattern Info
 
\tikzset{
pattern size/.store in=\mcSize, 
pattern size = 5pt,
pattern thickness/.store in=\mcThickness, 
pattern thickness = 0.3pt,
pattern radius/.store in=\mcRadius, 
pattern radius = 1pt}
\makeatletter
\pgfutil@ifundefined{pgf@pattern@name@_cxri9ypku}{
\pgfdeclarepatternformonly[\mcThickness,\mcSize]{_cxri9ypku}
{\pgfqpoint{0pt}{0pt}}
{\pgfpoint{\mcSize+\mcThickness}{\mcSize+\mcThickness}}
{\pgfpoint{\mcSize}{\mcSize}}
{
\pgfsetcolor{\tikz@pattern@color}
\pgfsetlinewidth{\mcThickness}
\pgfpathmoveto{\pgfqpoint{0pt}{0pt}}
\pgfpathlineto{\pgfpoint{\mcSize+\mcThickness}{\mcSize+\mcThickness}}
\pgfusepath{stroke}
}}
\makeatother

% Pattern Info
 
\tikzset{
pattern size/.store in=\mcSize, 
pattern size = 5pt,
pattern thickness/.store in=\mcThickness, 
pattern thickness = 0.3pt,
pattern radius/.store in=\mcRadius, 
pattern radius = 1pt}
\makeatletter
\pgfutil@ifundefined{pgf@pattern@name@_59829v9l6}{
\pgfdeclarepatternformonly[\mcThickness,\mcSize]{_59829v9l6}
{\pgfqpoint{0pt}{0pt}}
{\pgfpoint{\mcSize+\mcThickness}{\mcSize+\mcThickness}}
{\pgfpoint{\mcSize}{\mcSize}}
{
\pgfsetcolor{\tikz@pattern@color}
\pgfsetlinewidth{\mcThickness}
\pgfpathmoveto{\pgfqpoint{0pt}{0pt}}
\pgfpathlineto{\pgfpoint{\mcSize+\mcThickness}{\mcSize+\mcThickness}}
\pgfusepath{stroke}
}}
\makeatother

% Pattern Info
 
\tikzset{
pattern size/.store in=\mcSize, 
pattern size = 5pt,
pattern thickness/.store in=\mcThickness, 
pattern thickness = 0.3pt,
pattern radius/.store in=\mcRadius, 
pattern radius = 1pt}
\makeatletter
\pgfutil@ifundefined{pgf@pattern@name@_v5rl1o6ny}{
\pgfdeclarepatternformonly[\mcThickness,\mcSize]{_v5rl1o6ny}
{\pgfqpoint{0pt}{0pt}}
{\pgfpoint{\mcSize+\mcThickness}{\mcSize+\mcThickness}}
{\pgfpoint{\mcSize}{\mcSize}}
{
\pgfsetcolor{\tikz@pattern@color}
\pgfsetlinewidth{\mcThickness}
\pgfpathmoveto{\pgfqpoint{0pt}{0pt}}
\pgfpathlineto{\pgfpoint{\mcSize+\mcThickness}{\mcSize+\mcThickness}}
\pgfusepath{stroke}
}}
\makeatother

% Pattern Info
 
\tikzset{
pattern size/.store in=\mcSize, 
pattern size = 5pt,
pattern thickness/.store in=\mcThickness, 
pattern thickness = 0.3pt,
pattern radius/.store in=\mcRadius, 
pattern radius = 1pt}
\makeatletter
\pgfutil@ifundefined{pgf@pattern@name@_tsipkdw51}{
\pgfdeclarepatternformonly[\mcThickness,\mcSize]{_tsipkdw51}
{\pgfqpoint{0pt}{0pt}}
{\pgfpoint{\mcSize+\mcThickness}{\mcSize+\mcThickness}}
{\pgfpoint{\mcSize}{\mcSize}}
{
\pgfsetcolor{\tikz@pattern@color}
\pgfsetlinewidth{\mcThickness}
\pgfpathmoveto{\pgfqpoint{0pt}{0pt}}
\pgfpathlineto{\pgfpoint{\mcSize+\mcThickness}{\mcSize+\mcThickness}}
\pgfusepath{stroke}
}}
\makeatother
\tikzset{every picture/.style={line width=0.75pt}} %set default line width to 0.75pt        

\begin{tikzpicture}[x=0.75pt,y=0.75pt,yscale=-.75,xscale=.75]
%uncomment if require: \path (0,304); %set diagram left start at 0, and has height of 304

%Shape: Donut [id:dp4581708017222912] 
\draw   (587.19,91.08) .. controls (581.22,91.1) and (576.37,84.74) .. (576.34,76.88) .. controls (576.31,69.01) and (581.12,62.61) .. (587.09,62.59) .. controls (593.05,62.57) and (597.9,68.92) .. (597.93,76.79) .. controls (597.96,84.66) and (593.15,91.05) .. (587.19,91.08)(587.26,112.67) .. controls (569.38,112.74) and (554.82,96.75) .. (554.75,76.96) .. controls (554.69,57.17) and (569.13,41.07) .. (587.01,41) .. controls (604.89,40.93) and (619.45,56.91) .. (619.52,76.7) .. controls (619.59,96.49) and (605.14,112.59) .. (587.26,112.67) ;
%Shape: Donut [id:dp4222032094575152] 
\draw   (587.67,234.41) .. controls (581.71,234.44) and (576.85,228.08) .. (576.83,220.21) .. controls (576.8,212.34) and (581.61,205.95) .. (587.57,205.92) .. controls (593.53,205.9) and (598.39,212.26) .. (598.41,220.12) .. controls (598.44,227.99) and (593.63,234.39) .. (587.67,234.41)(587.75,256) .. controls (569.86,256.07) and (555.31,240.09) .. (555.24,220.3) .. controls (555.17,200.51) and (569.61,184.41) .. (587.49,184.34) .. controls (605.38,184.26) and (619.93,200.25) .. (620,220.04) .. controls (620.07,239.83) and (605.63,255.93) .. (587.75,256) ;
%Straight Lines [id:da49286025646452525] 
\draw    (576.12,186.29) -- (575.97,76.88) ;
%Shape: Rectangle [id:dp6077041635546958] 
\draw  [pattern=_0pa1t4syp,pattern size=6pt,pattern thickness=0.75pt,pattern radius=0pt, pattern color={rgb, 255:red, 0; green, 0; blue, 0}] (240.9,229.88) -- (240.82,205.99) -- (305.58,205.74) -- (305.66,229.62) -- cycle ;
%Shape: Rectangle [id:dp710605470060135] 
\draw  [pattern=_4dhc1fzc0,pattern size=6pt,pattern thickness=0.75pt,pattern radius=0pt, pattern color={rgb, 255:red, 0; green, 0; blue, 0}] (176.15,230.14) -- (176.06,206.25) -- (240.82,205.99) -- (240.9,229.88) -- cycle ;
%Shape: Rectangle [id:dp5154724996400297] 
\draw  [pattern=_cxri9ypku,pattern size=6pt,pattern thickness=0.75pt,pattern radius=0pt, pattern color={rgb, 255:red, 0; green, 0; blue, 0}] (283.59,86.38) -- (283.5,62.49) -- (348.26,62.23) -- (348.35,86.12) -- cycle ;
%Shape: Rectangle [id:dp5078938433669378] 
\draw   (305.66,229.62) -- (305.57,205.74) -- (370.33,205.48) -- (370.41,229.36) -- cycle ;
%Shape: Rectangle [id:dp36504280028888414] 
\draw   (154.08,86.89) -- (153.99,63) -- (218.75,62.74) -- (218.83,86.63) -- cycle ;
%Shape: Rectangle [id:dp5088069952398249] 
\draw   (218.83,86.63) -- (218.75,62.75) -- (283.51,62.49) -- (283.59,86.37) -- cycle ;
%Shape: Rectangle [id:dp9297975340001785] 
\draw  [pattern=_59829v9l6,pattern size=6pt,pattern thickness=0.75pt,pattern radius=0pt, pattern color={rgb, 255:red, 0; green, 0; blue, 0}] (111.39,230.4) -- (111.31,206.51) -- (176.07,206.25) -- (176.15,230.14) -- cycle ;
%Shape: Rectangle [id:dp5644709919258857] 
\draw  [pattern=_v5rl1o6ny,pattern size=6pt,pattern thickness=0.75pt,pattern radius=0pt, pattern color={rgb, 255:red, 0; green, 0; blue, 0}] (348.34,86.12) -- (348.26,62.23) -- (413.02,61.97) -- (413.1,85.86) -- cycle ;
%Shape: Rectangle [id:dp7985963028052492] 
\draw  [pattern=_tsipkdw51,pattern size=6pt,pattern thickness=0.75pt,pattern radius=0pt, pattern color={rgb, 255:red, 0; green, 0; blue, 0}] (283.99,205.82) -- (283.57,86.38) -- (305.17,86.29) -- (305.59,205.73) -- cycle ;
%Straight Lines [id:da8250144435949696] 
\draw    (598.79,220.12) -- (598.64,110.71) ;
%Shape: Rectangle [id:dp8515774364943438] 
\draw   (370.41,229.37) -- (370.33,205.48) -- (435.09,205.22) -- (435.17,229.11) -- cycle ;
%Shape: Rectangle [id:dp47072021962091015] 
\draw   (89.32,87.15) -- (89.24,63.26) -- (154,63) -- (154.08,86.89) -- cycle ;
%Shape: Rectangle [id:dp3833137090782335] 
\draw   (46.64,230.65) -- (46.55,206.77) -- (111.31,206.51) -- (111.39,230.39) -- cycle ;
%Shape: Rectangle [id:dp043394900580814966] 
\draw   (413.1,85.86) -- (413.01,61.97) -- (477.77,61.71) -- (477.86,85.6) -- cycle ;
%Straight Lines [id:da19291522323810661] 
\draw    (510,150) -- (523,150) -- (549,150) ;
\draw [shift={(551,150)}, rotate = 180] [color={rgb, 255:red, 0; green, 0; blue, 0 }  ][line width=0.75]    (10.93,-3.29) .. controls (6.95,-1.4) and (3.31,-0.3) .. (0,0) .. controls (3.31,0.3) and (6.95,1.4) .. (10.93,3.29)   ;

% Text Node
\draw (70,75) node  [font=\LARGE]  {$\dotsc $};
% Text Node
\draw (500,75) node  [font=\LARGE]  {$\dotsc $};
% Text Node
\draw (28,220) node  [font=\LARGE]  {$\dotsc $};
% Text Node
\draw (455,220) node  [font=\LARGE]  {$\dotsc $};
% Text Node
\draw (624,61.4) node [anchor=north west][inner sep=0.75pt]    {$A_{1}$};
% Text Node
\draw (626.62,215.57) node [anchor=north west][inner sep=0.75pt]    {$A_{2}$};
% Text Node
\draw (525,153.4) node [anchor=north west][inner sep=0.75pt]    {$\pi $};
% Text Node
\draw (606,139.4) node [anchor=north west][inner sep=0.75pt]    {$h$};
% Text Node
\draw (231,233.4) node [anchor=north west][inner sep=0.75pt]    {$\widetilde{A_{2}}$};
% Text Node
\draw (271,34) node [anchor=north west][inner sep=0.75pt]    {$\widetilde{A_{1}}$};

\end{tikzpicture}

\caption{The shaded area is the disk constructed in Lemma~\ref{spectacles_lmm}. In this figure, $m_i = 2$ and $n_i = 3$.}\label{spectacles_figure}
\end{figure}
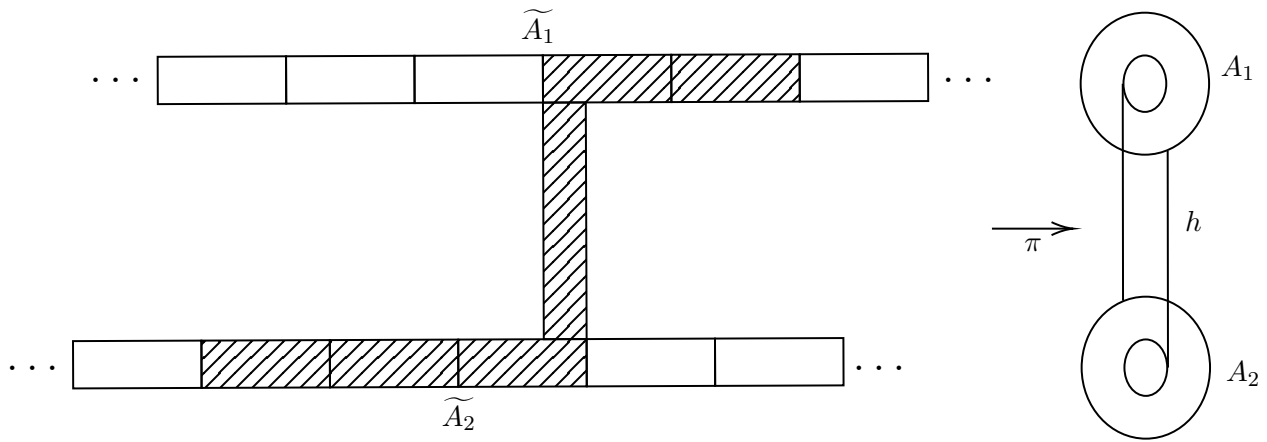

\begin{proof}[{\bf{\emph{Proof of Theorem~\ref{convexity_thm}}}}]
The proof succeeds by successive approximations. First, by definition, two Ahlfors currents $T_1$ and $T_2$ are approximated by disks. These disks are approximated from within by slightly smaller disks. By Proposition~\ref{reduction_to_surface_case_prp}, the latter are then approximated by regions in an irreducible $J$-holomorphic curve. Finally, any convex combination of these regions is approximated by disks by Lemma~\ref{spectacles_lmm}. The appropriate bounds on length and area (i.e.~that the lengths of the boundaries divided by the areas tends to $0$) are preserved throughout. As a limit of limits of currents of finite mass is a limit in the weak topology \cite[Corollary~7.1]{Fed60}, a diagonal sequence of these disks tends to a convex combination of $T_1$ and $T_2$.
\end{proof}

The proof of Example~\ref{extreme_points_eg} is already substantially contained in \cite{Har83}, so we only briefly describe the additional step below.

\begin{proof}[{\bf{\emph{Proof of Example~\ref{extreme_points_eg}}}}]
By \cite[Lemma~18]{Har83}, the extreme points of the convex set of closed positive currents in the fiber class are the fibers (and thus Ahlfors currents) and every closed positive current is in the closed convex hull of the fibers. Therefore, by Theorem~\ref{convexity_thm}, every closed positive current in the fiber class is an Ahlfors current. Thus, the space of Ahlfors currents in the fiber class equals the space of closed positive currents in the fiber class. As just mentioned, the extreme points of the latter set are the fibers.
\end{proof}

\vspace{.3in}

{\it Department of Mathematics, Stony Brook University, Stony Brook, NY 11794\\
spencer.cattalani@stonybrook.edu}

\bibliography{references}

\end{document}